%% file: paper-2024-es_hr.tex
\documentclass[final,3p,times]{elsarticle}

\usepackage[hidelinks]{hyperref}
\usepackage{amsmath,amssymb}
\usepackage{mathtools}	
\usepackage{stmaryrd} 
\usepackage{bm}		  
\usepackage{accents}
\usepackage{cancel}
\usepackage[toc,page]{appendix}
\usepackage{subfigure}
\usepackage{booktabs}
\usepackage{epstopdf}
\usepackage{pgfplots}
\usepackage{tikz}
\usetikzlibrary{patterns,decorations.pathreplacing, pgfplots.fillbetween, intersections}
\usepackage{nicefrac}
\usepackage[T1]{fontenc}
\usepackage[utf8]{inputenc}

\usepackage[ruled]{algorithm2e}
\DontPrintSemicolon

\newcommand{\avg}[1]{\left\{\hspace*{-1pt}\left\{#1\right\}\hspace*{-1pt}\right\}}
\newcommand{\jump}[1]{\ensuremath{\left\llbracket #1 \right\rrbracket}}

\newcommand\iprodN[1]{\left\langle #1\right\rangle_{\!N}}		

\renewcommand\vec[1]{\accentset{\,\rightarrow}{#1}}
\newcommand\spacevec[1]{\accentset{\,\rightarrow}{#1}}		
\newcommand\statevec[1]{\mathbf #1}					
\newcommand\statevecGreek[1]{\boldsymbol #1}			
\newcommand\acclrvec[1]{\accentset{\,\leftrightarrow}{#1}}	
\newcommand\bigstatevec[1]{\acclrvec{{\mathbf #1}}}	
\newcommand\bigcontravec[1]{\acclrvec{\tilde{\mathbf #1}}} 	

\newcommand\normalmatrix[1]{\underline{ #1}}		
\newcommand\dS{\,\operatorname{dS} }

\newtheorem{theorem}{Theorem}
\newtheorem{lemma}[theorem]{Lemma}
\newdefinition{remark}{Remark}
\newproof{proof}{Proof}

\journal{Journal of Computational Physics}

\newcommand{\orcid}[1]{\href{https://orcid.org/#1}{\includegraphics[width=10pt]{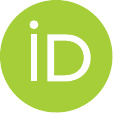}}}

\begin{document}

\begin{frontmatter}

\title{Entropy stable hydrostatic reconstruction schemes for shallow water systems}

\author[liu]{Patrick Ersing\corref{cor1}\orcid{0009-0005-3804-5380}}
\ead{patrick.ersing@liu.se}
\cortext[cor1]{Corresponding author}

\author[dlr]{Sven Goldberg\orcid{0009-0008-4419-3162}}

\author[liu]{Andrew R. Winters\orcid{0000-0002-5902-1522}}

\affiliation[liu]{organization={Linköping University, Department of Mathematics},
            city={Linköping},
            postcode={58183}, 
            country={Sweden}}
        
\affiliation[dlr]{organization={German Aerospace Center (DLR), Institute of Software Technology, Department of High-Performance Computing},
	addressline={Linder Höhe}, 
	city={Cologne},
	postcode={51147}, 
	country={Germany}}

\begin{abstract}
	In this work, we develop a new hydrostatic reconstruction procedure  to construct well-balanced schemes for one and multilayer shallow water flows, including wetting and drying.
	Initially, we derive the method for a path-conservative finite volume scheme and combine it with entropy conservative fluxes and suitable numerical dissipation to preserve an entropy inequality in the semi-discrete case.
	We then combine the novel hydrostatic reconstruction with a collocated nodal split-form discontinuous Galerkin spectral element method, extending the method to high-order and curvilinear meshes.
	The high-order method incorporates an additional positivity-limiter and is blended with a compatible subcell finite volume method to maintain well-balancedness at wet/dry fronts.
	We prove entropy stability, well-balancedness, and positivity-preservation for both methods.
	Numerical results for the high-order method validate the theoretical findings and demonstrate the robustness of the scheme.
\end{abstract}


\begin{keyword}
Multilayer shallow water equations \sep Discontinuous Galerkin method \sep Well-balanced \sep Wetting and drying \sep Entropy stability \sep Positivity-preserving

\MSC 65M12 \sep 65M20 \sep 65M70 \sep 76M22 \sep 35L60 \sep 86A05
\end{keyword}

\end{frontmatter}

\section{Introduction}
The shallow water equations (SWE) are a widely used model for free-surface flows, where horizontal length scales dominate the vertical scales as encountered in numerous environmental applications. The equations are derived from depth-averaging the incompressible Navier-Stokes equations and are given in one-spatial dimension by
\begin{equation}
	\left\{
	\begin{aligned}
		&\partial_t h + \partial_x hv = 0,\\
		&\partial_t hv + \partial_x hv^2 = -gh\partial_x\left(h + b\right),
	\end{aligned}\right.
	\label{eq:system_swe}
\end{equation}
where $h \coloneqq h(x,t)$ denotes the water height, $v \coloneqq v(x,t)$ the velocity, $g$ the gravitational acceleration and $b \coloneqq b(x)$ a prescribed bottom topography.

An extension for stratified flows is given by the multilayer shallow water equations (ML-SWE) \cite{frings2012adaptive}, that describe the flow of immiscible layers with different densities.
\begin{equation}
	\left\{
	\begin{aligned}			
		&\partial_t h_m + \partial_x h_mv_m = 0,\\
		&\partial_t h_mv_m + \partial_x h_mv_m^2 = -gh_m\partial_x \bigg(b + \sum\limits_{k\geq m}h_k + \sum\limits_{k<m}\sigma_{km}h_k \bigg),
	\end{aligned}\right.
	\label{eq:system_multilayer}
\end{equation}
where $m=1,..,M$ denotes the respective layer, $\rho_m$ the layer density and $\sigma_{km} \coloneqq \frac{\rho_k}{\rho_m}$ is the quotient of layer densities. 
We adopt the convention to count layers from top to bottom, such that $0 < \rho_1 < ... < \rho_M$, as depicted in Figure~\ref{fig:multilayer_system}.
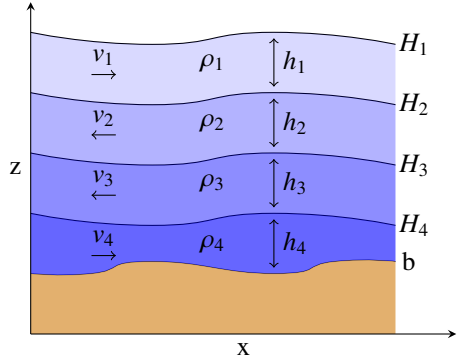
\begin{figure}
	\centering
	\input{MLSWE_Visu.tex}
	\caption{Illustration of the multilayer shallow water system with four layer and bottom topography.}
	\label{fig:multilayer_system}
\end{figure}
Each layer of the model corresponds to a set of SWEs with an additional coupling term.
Accordingly, the model reduces to the SWE for $M=1$.
Despite this close resemblance, the stratification in the ML-SWE introduces additional difficulties for the numerical treatment, as the system is only conditionally hyperbolic and introduces nonconservative products through the coupling terms between layers.

An important property for numerical approximations of shallow water models is the preservation of steady state solutions, most importantly the so called lake-at-rest condition given for the SWE and the ML-SWE by
\begin{equation}
	h\partial_x(h+b) = 0, \quad v = 0
	\label{eq:lake-at-rest_swe},
\end{equation}
\begin{equation}
	h_m\partial_x\left(b + \sum\limits_{k\geq j}h_k + \sum\limits_{k<m}\sigma_{km}h_k\right) = 0, \quad v_m = 0, \quad m=1,...,M,
	\label{eq:lake-at-rest_multilayer}
\end{equation}
respectively.

Preservation of this steady state is important as solutions to these equations, typically resemble small perturbations from the lake-at-rest. Numerical schemes that satisfy this lake-at-rest condition are said to be well-balanced. 
Violations of this condition lead to artificial waves on the order of the grid spacing that pollute the physically meaningful solution.
The preservation of these steady states is especially challenging in the presence of wet/dry transitions, where the hydrostatic pressure may become discontinuous at the transition point.

For the SWE well-balancedness is commonly achieved by a hydrostatic reconstruction procedure \cite{audusse2004fast}.
In this approach, the bottom topography is approximated by a piecewise linear reconstruction to recover a continuous water height and bottom topography at interfaces and, thereby, obtain the well-balanced property.
For the SWE this technique is well-understood and many improved reconstructions have been proposed to deal with possible reflections or the waterfall effect as discussed in \cite{xia2017efficient}.
Furthermore, in \cite{fu2022high, xing2013positivity} the technique has been successfully applied to obtain well-balanced discontinuous Galerkin (DG) methods.

For the ML-SWE, on the other hand, the application of the hydrostatic reconstruction becomes more difficult.
Through the introduction of density ratios $\sigma$, a direct exchange of bottom topography and water height with regards to the hydrostatic pressure is no longer possible.
Instead, contemporary approaches typically reconstruct each layer independently to retain well-balancedness, see \cite{martinez20201d, dudzinski2013well, castro2005numerical}.
However, the physical motivation and implications behind these reconstructions often remain unclear.
As observed for the SWE in \cite{audusse2004fast,chen2017new} the reconstruction procedure typically violates entropy conservation, which needs to be balanced with sufficient dissipation.
This is important to ensure correct shock speeds according to the Rankine-Hugoniot conditions and nonlinear stability of the scheme \cite{merriam1989entropy}.
While entropy stable hydrostatic reconstruction methods are common for the SWE, to the best of our knowledge, there are no results about entropy stability for the ML-SWE. We aim to fill this gap in the current paper by developing novel entropy stable hydrostatic reconstruction methods.

In recent years, several works have considered positivity-preserving DG methods to obtain high-order approximations for the SWE \cite{fu2022high,wintermeyer2018entropy,xing2013positivity,wu2021entropy, bonev2018discontinuous, ranocha2017shallow}.
In addition to their favorable dispersion and dissipation properties, the polynomial representation allows for a natural approximation of the bottom topography and can be used to create high-order approximations.
In the presence of wet/dry transitions, positivity preservation is often achieved through positivity-limiting \cite{xing2014survey} or a subcell-limiting procedure \cite{wu2021high}. 
However, the high-order nature of these methods incurs additional difficulty to preserve steady states if wet/dry transitions occur.
As a consequence, it is a common limitation that methods require the wet/dry transition to be aligned with element interfaces to achieve well-balancedness, as reported in \cite{wu2021entropy, wintermeyer2018entropy, fu2022high}. 
A solution to this was proposed in \cite{bonev2018discontinuous}, where the authors obtained well-balancedness by introducing a first-order approximation at partially dry elements.

In this work we present a new hydrostatic reconstruction method that can be applied for one and multilayer shallow water equations. 
We begin with a specific formulation of the hydrostatic pressure term in system \eqref{eq:system_multilayer} to obtain a direct representation of the lake-at-rest condition in the nonconservative term.
The reconstruction procedure is constructed in a special way and combined with two-point entropy stable fluxes to obtain a method that is positivity-preserving, well-balanced, and entropy stable for both SWE and ML-SWE.
In contrast to previous hydrostatic reconstruction methods, we present a novel mechanism to obtain well-balancedness. Furthermore, we show that a specific reconstruction choice yields a discretization without introducing any additional nonconservative terms. 

We first derive the new hydrostatic reconstruction method from the perspective of a path-conservative, first-order finite volume (FV) method in Section \ref{sec:FV_scheme}, where we show entropy stability, well-balancedness, and positivity-preservation.
To clarify the mechanisms behind the new reconstruction, we first demonstrate these properties for the SWE, before we consider the ML-SWE.
After introducing the low-order FV scheme, we demonstrate in Section \ref{sec:DG_method} how the hydrostatic reconstruction method can be extended to a high-order DG approximation on unstructured curvilinear meshes.
We discuss how to construct the high-order method by using a flux-differencing formulation in combination with shock-capturing and a positivity-limiter to preserve positivity, well-balancedness, and entropy stability.
Numerical results in Section \ref{sec:results} verify the theoretical results on several numerical test cases for both SWE and ML-SWE.
We draw conclusions in the final section.

\section{Finite volume hydrostatic reconstruction scheme}\label{sec:FV_scheme}
In this section we describe the novel hydrostatic reconstruction scheme that is based on a path-conservative first-order FV method. 
Therefore, we introduce the following general notation for both shallow water systems \eqref{eq:system_swe} and \eqref{eq:system_multilayer}
\begin{equation}
	\statevec{u}_t + \partial_x \statevec{f}(\statevec{u}) + \statevecGreek{\phi}(\statevec{u})\circ\partial_x\statevec{r}(\statevec{u}) = 0,
	\label{eq:balance_law_continuous}
\end{equation}
with the conserved variables $\statevec{u} = (\statevec{u}^T_1,...,\statevec{u}^T_M)^T$, the conservative flux $\statevec{f} = (\statevec{f}^T_1, ..., \statevec{f}^T_M)^T$ and the nonconservative term given by the Hadamard product of $\statevecGreek{\phi} = (\statevecGreek{\phi}^T_1, ..., \statevecGreek{\phi}^T_M)^T$ and $\statevec{r} = (\statevec{r}^T_1, ..., \statevec{r}^T_M)^T$. In each layer $m$ the components of these quantities are
\begin{equation}
	\begin{aligned}
		\statevec{u}_m = \begin{pmatrix} h_m \\ h_mv_m\end{pmatrix}, \quad
		\statevec{f}_m = \begin{pmatrix} h_mv_m \\ h_mv_m^2 \end{pmatrix}, \quad 
		\statevecGreek{\phi}_m = g\begin{pmatrix} 0 \\ h_m \end{pmatrix}, \quad
		\statevec{r}_m = \begin{pmatrix} 0 \\ b + \sum\limits_{k\geq m}h_k + \sum\limits_{k<m}\sigma_{km}h_k \end{pmatrix},
	\end{aligned}
\end{equation}
where $m=1,...,M$ for the ML-SWE and $m=1$ for the SWE. In this formulation the hydrostatic pressure term has been peeled of as a nonconservative term, while the conservative flux consists only of advective terms. This allows for a direct recovery of the pressure balance with respect to the lake-at-rest condition \eqref{eq:lake-at-rest_multilayer} in the nonconservative term to obtain well-balancedness in the discrete case.

\subsection{First-order path-conservative FV scheme}
We initially develop a first-order accurate, path-conservative FV scheme on one-dimensional domains.
Therefore, we divide the domain $\Omega$ into $K$ non-overlapping elements $E_i = [x_{i-\frac{1}{2}}, x_{i+\frac{1}{2}}]$ with size $\Delta x$.
To derive the FV scheme, we integrate \eqref{eq:balance_law_continuous} over a single element $E$ to have
\begin{equation}
	\begin{aligned}
		\int_E \statevec{u}_t\,\operatorname{dx} &= -\int_E \partial_x\statevec{f} + \statevecGreek{\phi}\circ\partial_x\statevec{r}\,\operatorname{dx}.\\
	\end{aligned}	
\end{equation}
We then apply Gauss' Law to obtain a surface integral for the conservative flux and replace the advective flux $\statevec{f}$ at element interfaces with a numerical flux $\statevec{f}^*$, that is uniquely defined at discontinuities.
At each interface we introduce an outward pointing normal vector $\spacevec{n}$. 
Even though we only consider the one-dimensional setting for now, we want to keep the discussion general for extension to a formulation on two-dimensional unstructured meshes.
So, we denote the primary element at each interface as $\boldsymbol{-}$ and the secondary element as $\boldsymbol{+}$, ensuring the normal vector points from $\boldsymbol{-}$ to $\boldsymbol{+}$ with the relation
\begin{equation}
	\spacevec{n} = \spacevec{n}^- = - \spacevec{n}^+,
\end{equation}
where in the one-dimensional setting $\spacevec{n}$ reduces to a scalar $\spacevec{n} \in \{-1,1\}$.

Furthermore, we introduce the following notation to denote arithmetic averages and jumps at interfaces as
\begin{equation}
	\avg{\cdot} = \frac{(\cdot)^+ + (\cdot)^-}{2}, \quad
	\jump{\cdot} = (\cdot)^+ - (\cdot)^-. 
\end{equation}
To obtain a path-conservative scheme analogous to \cite{pares2006numerical, franquet2012runge} we replace the nonconservative term with a Borel measure, that corresponds to classical integration for smooth solutions, but requires evaluation as a path integral at discontinuities. 
Accordingly, the volume integral of the nonconservative term remains unchanged, while we introduce a surface numerical nonconservative term at the interface to obtain the discretization
\begin{equation}
	\statevec{u}_t = -\frac{1}{\Delta x} \oint_{\partial E} \left(\statevec{f^*} + \left(\statevecGreek{\phi}\circ \statevec{r}\right)^{\diamond}\right)\cdot \vec{n} \operatorname{dS}
	-\frac{1}{\Delta x}\int_E \statevecGreek{\phi}\circ\partial_x\statevec{r}\,\operatorname{dx}.
	\label{eq:fv_scheme}
\end{equation}
This surface numerical nonconservative term satisfies the consistency condition 
\begin{equation}
	\left(\statevecGreek{\phi}\circ \statevec{r}\right)^{\diamond}(\statevec{u},\statevec{u}) = 0,
\end{equation}
as well as the path-conservation property \cite{pares2006numerical}. That is, when summed over each interface
\begin{equation}
	\left(\statevecGreek{\phi}\circ\statevec{r}\,\right)^{\diamond}(\statevec{u}^-,\statevec{u}^+)\cdot\spacevec{n}^{\,-}
	+
	\left(\statevecGreek{\phi}\circ\statevec{r}\,\right)^{\diamond}(\statevec{u}^+,\statevec{u}^-)\cdot\spacevec{n}^{\,+} 
	=
	\int\limits_0^1 {\statevec{A}}^{NC}(\upsilon(s;\statevec{u}^-,\statevec{u}^+))\partial_s \upsilon(s;\statevec{u}^-,\statevec{u}^+)\operatorname{ds}.
	\label{eq:path_conservative_property}
\end{equation}
Here, the right hand side denotes a path integral determined by the flux Jacobian of the nonconservative subsystem $\statevec{A}^{NC}$, where $\upsilon$ represents a path connecting adjacent states $\statevec{u}^-,\statevec{u}^+$.
In this work we consider the numerical nonconservative term from \cite{ersing2024entropy} satisfying these conditions, which is given by
\begin{equation}
	\left(\statevecGreek{\phi}\circ\statevec{r}\,\right)^{\diamond}(\statevec{u}^-,\statevec{u}^+) = \frac{\statevecGreek{\phi}^-}{2} \circ \jump{\statevec{r}}.
	\label{eq:nonconservative_flux}
\end{equation}

From the FV discretization \eqref{eq:fv_scheme}, we now want to recover a well-balanced and entropy stable scheme. 
In the fully wet setting it is possible to achieve both properties directly through a specific combination of the conservative flux and nonconservative term as demonstrated in \cite{ersing2024entropy, fjordholm2012energy, winters2015comparison}.
However, to ensure well-balancedness and positivity preservation under wet-dry transitions, instead we introduce a novel hydrostatic reconstruction.

\subsection{Hydrostatic Reconstruction Method}
Inspired by the subcell interpretation from \cite{chen2017new}, we derive the new piecewise linear hydrostatic reconstruction. At each interface we introduce infinitesimal subcells $\hat{E}^{\epsilon,\pm}_{i+\nicefrac{1}{2}}$ in adjacent elements. Consequently, each element is now composed of three subcells
\begin{equation}
	E_i \coloneqq \hat{E}^{\epsilon,+}_{i-\frac{1}{2}} \cup {E}^{\epsilon}_{i} \cup \hat{E}^{\epsilon,-}_{i+\frac{1}{2}} =
	\left[ x_{i-\frac{1}{2}}, x_{i-\frac{1}{2}} + \epsilon\right]
	\cup
	\left[ x_{i-\frac{1}{2}} + \epsilon, x_{i+\frac{1}{2}} - \epsilon\right]
	\cup
	\left[ x_{i+\frac{1}{2}} - \epsilon, x_{i+\frac{1}{2}}\right],
	\label{eq:definition_subcells}
\end{equation}
where $\epsilon$ represents some infinitely small number.
While $\epsilon$ is required for its definition, the actual reconstruction method does not depend on the value of $\epsilon$ and therefore no explicit value needs to be assigned.

Within each infinitesimal subcell $\hat{E}^{\epsilon,\pm}_{i + \nicefrac{1}{2}}$ we then define a linear reconstruction of the bottom topography $\hat{b}^{\pm}_{i+\nicefrac{1}{2}}(x)$ between the inner value $b_{i+\nicefrac{1}{2}}^{\pm} := b(x_{i+\nicefrac{1}{2}}\pm\epsilon)$ and some boundary values $b_{ i+\nicefrac{1}{2}, \epsilon}^{\pm}$ located at the interface.
This yields the piecewise linear and possibly discontinuous reconstruction of the bottom topography
\begin{equation}
	b_{\epsilon}(x) \coloneqq 
	\begin{cases}
		b_i & \text{if} \quad x \in E_{i}^{\epsilon}, \\
		\hat{b}^{\pm}_{i+\frac{1}{2}}(x) & \text{if} \quad x \in \hat{E}^{\epsilon,\pm}_{i+\frac{1}{2}},
	\end{cases}
\end{equation}
where $\hat{b}^{\pm}_{i+\nicefrac{1}{2}}(x)$ denotes the linear reconstruction of the bottom topography in the respective infinitesimal subcells $\hat{E}^{\epsilon,\pm}_{i + \nicefrac{1}{2}}$ and $b_i$ denotes the original bottom topography value in cell $E^{\epsilon}_i$. In Figure \ref{fig:reconstruction_visu} the reconstruction is illustrated for a wet/dry transition at an interface with two adjacent finite volume cells.
The choice of the boundary values $b_{ i\pm\nicefrac{1}{2}, \epsilon}^{\pm}$ now specifies the hydrostatic reconstruction, offering the flexibility to recover different well-known hydrostatic reconstruction schemes as demonstrated in \cite{chen2017new}.
A common strategy to obtain the well-balanced property is to choose a continuous reconstruction of the bottom topography to ensure that the nonconservative term vanishes, while the conservative fluxes are balanced under lake-at-rest conditions.
However, enforcing continuous bottom topographies by the reconstruction can lead to entropy violations, that must be countered by suitable dissipation \cite{chen2017new, audusse2004fast}.

Instead, we introduce a novel reconstruction of the bottom topography, which recovers a continuous bottom topography in fully wet scenarios or a discontinuous reconstruction at dry states. 
In combination with entropy conservative fluxes, the new reconstruction will be shown to conserve entropy for the SWE and allows us to derive entropy stable reconstructions for both SWE and ML-SWE.
Furthermore, with the nonconservative formulation of the pressure term in \eqref{eq:system_swe} and \eqref{eq:system_multilayer}, the new approach introduces no additional reconstruction terms. This improves efficiency of the method and makes the required adaptations for existing FV codes very simple.

The new reconstruction is defined by setting the boundary values for the bottom topography at each interface to
\begin{equation}
	b_{\epsilon}^{\pm} \coloneqq \min\left(H_1^{\pm},\max(b^-,b^+)\right),
	\label{eq:reconstruction_bottom}
\end{equation}
where we introduced the total layer height $H_m \coloneqq \sum_{k\geq m} h_k + b$ and the superscript $\pm$ in this case denotes the (unreconstructed) values adjacent to each interface.
In each element we then reconstruct the total layer heights to ensure positivity  
\begin{equation}
	H_{m,\epsilon}(x) \coloneqq \max\left(H_{m}, b_{\epsilon}(x)\right), \quad \forall x \in E_i.
	\label{eq:reconstruction_layerheight}
\end{equation}
We recover the water height in each layer from
\begin{equation}
	h_{m,\epsilon}(x) =
	\begin{cases}
		H_{M,\epsilon}(x) - b_{\epsilon}(x) & \text{if}\quad m=M,\\
		H_{m,\epsilon}(x) - H_{(m+1),\epsilon}(x) &\text{otherwise}.
	\end{cases}
	\label{eq:reconstruction_waterheight}
\end{equation}

Note that by taking the maximum in \eqref{eq:reconstruction_layerheight}, for the ML-SWE both $H_{m,\epsilon}(x)$ and $h_{m,\epsilon}(x)$ are only piecewise linear even within each infinitesimal subcell, which is important in the proof of Lemma \ref{lemma:vanishing_hr_terms} below.

After the hydrostatic reconstruction procedure \eqref{eq:reconstruction_bottom} -- \eqref{eq:reconstruction_waterheight} the reconstructed conservative variables are
\begin{equation}
	\statevec{u}_{m,\epsilon}(x) \coloneqq \begin{pmatrix}
		h_{m,\epsilon}(x) \\ h_{m,\epsilon}(x) v_m
	\end{pmatrix},
	\quad \forall x \in E_i.
	\label{eq:reconstruction_conservative_variables}
\end{equation}
We evaluate \eqref{eq:fv_scheme} in terms of the reconstructed variables to obtain a preliminary version of the hydrostatic reconstruction scheme
\begin{equation}
	\statevec{u}_t = -\frac{1}{\Delta x} 	\oint\limits_{\partial E} ({\statevec{f}}_{\epsilon}^* + ({\statevecGreek{\phi}}\circ 	{\statevec{r}})^{\diamond}_{\epsilon})\cdot \vec{n} \operatorname{dS}
	+ 	\frac{1}{\Delta x} \int\limits_{E}\statevecGreek{\phi_{\epsilon}}\circ{\partial_x\statevec{r}_{\epsilon}}\,\operatorname{dx},
	\label{eq:HR_scheme_with_vol_terms}
\end{equation}
where $\statevec{f}_{\epsilon}^*\coloneqq\statevec{f}^*(\statevec{u}_{\epsilon}^-,\statevec{u}_{\epsilon}^+)$ and $({\statevecGreek{\phi}}\circ {\statevec{r}})^{\diamond}_{\epsilon}\coloneqq({\statevecGreek{\phi}}\circ {\statevec{r}})^{\diamond}(\statevec{u}_{\epsilon}^-,\statevec{u}_{\epsilon}^+)$.

Compared to standard FV formulations the piecewise linear reconstruction of conservative variables and bottom topography generates additional volume contributions from the nonconservative term, as commonly encountered in other hydrostatic reconstruction schemes, e.g., in \cite{audusse2004fast, chen2017new, dong2023surface}.
However, we next demonstrate that the new reconstruction in combination with the nonconservative pressure formulation in \eqref{eq:system_multilayer} ensures that these additional terms vanish and no additional terms need to be introduced for the reconstruction.
\begin{lemma}
	\label{lemma:vanishing_hr_terms}
	The hydrostatic reconstruction scheme \eqref{eq:HR_scheme_with_vol_terms} simplifies to
	\begin{equation}
		\statevec{u}_t = -\frac{1}{\Delta x} 	\oint_{\partial E} ({\statevec{f}}_{\epsilon}^* + ({\statevecGreek{\phi}}\circ 	{\statevec{r}})^{\diamond}_{\epsilon})\cdot \vec{n} \operatorname{dS}
		\label{eq:hydrostatic_reconstruction_scheme}
	\end{equation} 
	as the volume integral of the nonconservative term vanishes for both SWE \eqref{eq:system_swe} and ML-SWE \eqref{eq:system_multilayer}, when the new hydrostatic reconstruction is used.
\end{lemma}
\begin{proof}
	To show that the nonconservative volume terms vanish we split the volume integral from \eqref{eq:HR_scheme_with_vol_terms} into the subcell contributions and use that the integral at inner subcells $E^{\epsilon}_i$ vanishes as the approximation remains constant
	\begin{equation}
		\int\limits_{E_i} \statevecGreek{\phi}_{\epsilon}\circ \partial_x{\statevec{r}_{\epsilon}} \, \operatorname{dx} = 
		\int\limits_{\hat{E}^{\epsilon+}_{i-\nicefrac{1}{2}}} \statevecGreek{\phi_{\epsilon}}\circ 	\partial_x{\statevec{r}_{\epsilon}}\,\operatorname{dx} 
		+ \int\limits_{\hat{E}^{\epsilon-}_{i+\nicefrac{1}{2}}} \statevecGreek{\phi}_{\epsilon} \circ 	\partial_x{\statevec{r}_{\epsilon}}\,\operatorname{dx}.
	\end{equation}

	For the SWE our claim follows directly as the reconstruction preserves the total water height $H$ within each element.
	Hence, we obtain that $\statevec{r}_{\epsilon}(x) = H_{\epsilon}(x)$ remains constant in each subcell and the volume integral vanishes from $\partial_x{\statevec{r}_{\epsilon}}=0$.
	
	For the ML-SWE the situation is more complicated due to additional layer coupling terms.
	Although the total water height $H_1$ is preserved by the reconstruction, we now have that $\statevec{r}_{\epsilon}(x) \neq H_{\epsilon}(x)$, due to the additional quotient of layer densities and $\statevec{r}_{\epsilon}(x)$ is no longer constant.
	To examine the volume term we evaluate the integral at a single interface.
	By taking the maximum in \eqref{eq:reconstruction_layerheight}, the reconstructed layer heights $h_{m,\epsilon}(x)$ from \eqref{eq:reconstruction_waterheight} become non-smooth at the point of their respective wet/dry transition, see Figure~\ref{fig:reconstruction_visu}.
	To account for this we split the integral into further subintegrals $\epsilon_m$ between these transition points
	\begin{equation}
		\int\limits_{\hat{E}^{\epsilon-}_{i+\nicefrac{1}{2}}} \statevecGreek{\phi_{\epsilon}}\circ 	\partial_x{\statevec{r}_{\epsilon}}\,\operatorname{dx} 
		=
		\int\limits_{\epsilon_1} \statevecGreek{\phi}_{\epsilon} \circ 	\partial_x{\statevec{r}_{\epsilon}}\,\operatorname{dx}
		+...+
		\int\limits_{{\epsilon_M}} \statevecGreek{\phi}_{\epsilon} \circ 	\partial_x{\statevec{r}_{\epsilon}}\,\operatorname{dx}.
	\end{equation}

	From this splitting, we can examine each subintegral separately. That is, in each subintegral $\epsilon_m$, only layer $m$ is affected by the bottom reconstruction, and only $h_{m,\epsilon}$ varies on the respective interval $\epsilon_m$.
	Due to the linear nature of the bottom reconstruction all layers below $m$ must be dry, so $h_{k,\epsilon}=0$ for $k>m$ and from \eqref{eq:reconstruction_waterheight} it follows that $h_{m,\epsilon}+b_{\epsilon}$ is constant.
	On the other hand, layers $k<m$ are not affected by the reconstruction and $h_{k,\epsilon}$ is $\text{constant}$ for $k>m$. 
	Combining these properties this means that $\statevec{r}_{m,\epsilon}$ is constant for each subintegral $\epsilon_m$ such that
	\begin{equation}
		\int\limits_{{\epsilon_m}} \statevecGreek{\phi}_{m,\epsilon} \circ 	\partial_x{\statevec{r}_{m,\epsilon}}\,\operatorname{dx}
		= \int\limits_{{\epsilon_m}} g \begin{pmatrix} 0 \\ h_{m,\epsilon} \end{pmatrix}
		\circ 
		\partial_x\begin{pmatrix} 0 \\ b_{\epsilon} + \sum\limits_{k\geq m}h_{k,\epsilon} + \sum\limits_{k<m}\sigma_{km}h_{k,\epsilon}\end{pmatrix}
		= \statevec{0},
	\end{equation}
	and it follows that $\int\limits_{E_i} \statevecGreek{\phi}_{\epsilon}\circ \partial_x{\statevec{r}_{\epsilon}} \, \operatorname{dx}=0$.
	\qed
\end{proof}

\begin{figure}
	\subfigure[]{
		\input{ML_norec.tex}
	}
	\hspace{0.15\textwidth}
	\subfigure[]{
		\input{ML_reconstructed.tex}
	}
	\caption{Illustration of the bottom topography and layer heights for an interface located at $x_{i+\nicefrac{1}{2}}$ with adjacent finite volume cells $E_i$ and $E_{i+1}$} before (a) and after the reconstruction procedure (b) at a wet/dry transition.
	\label{fig:reconstruction_visu}
\end{figure}

\begin{figure}[htb]
	\subfigure[wet]{
		\input{Wet_case.tex}}
	\hfill
	\subfigure[dry]{
		\input{Dry_Case.tex}}
	\hfill
	\subfigure[partially dry]{
		\input{Partial_Dry_case.tex}}
	\caption{Illustration of the bottom topography reconstruction for the different configurations (a) wet, (b) dry and (c) partially dry at an interface with two adjacent finite volume cells, where superscripts $\boldsymbol{-}$ and $\boldsymbol{+}$ denote the respective values from the left and right cell. The subscript $\epsilon$ denotes the reconstructed bottom topography \eqref{eq:reconstruction_bottom}.}
	\label{fig:layer_configurations}
\end{figure}
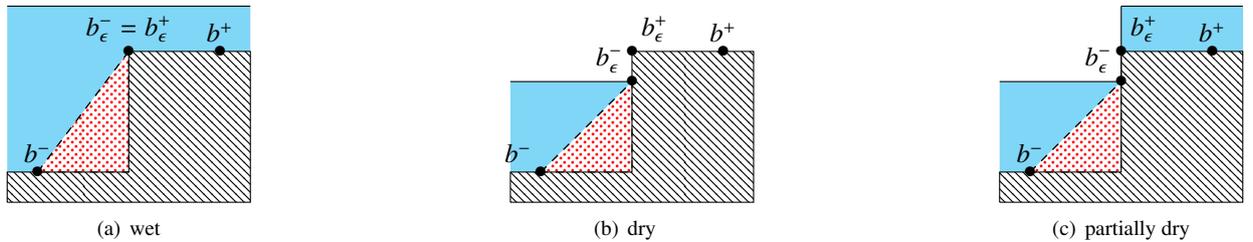

With the result of Lemma \ref{lemma:vanishing_hr_terms} we can omit the volume integral in \eqref{eq:HR_scheme_with_vol_terms} and only consider the simplified approximation~\eqref{eq:hydrostatic_reconstruction_scheme} in the forthcoming analysis.

Next, we demonstrate well-balancedness and entropy stability for the new hydrostatic reconstruction.
Both properties are first demonstrated on the SWE before we generalize to the more involved case of the ML-SWE.\\

\subsection{Well-balancedness}\label{sec:well-balancedness}
From the lake-at-rest conditions \eqref{eq:lake-at-rest_swe} and \eqref{eq:lake-at-rest_multilayer} it is clear that well-balancedness may be achieved either through a balance in the hydrostatic pressure or if the layer height vanishes.
While the first strategy is widely used to achieve well-balanced schemes, to the best of our knowledge, the second strategy has not been considered in previous hydrostatic reconstructions.
As such, we leverage this additional flexibility and choose the reconstruction to either satisfy $h_{m,\epsilon}=0$ or $\statevec{r}_{m,\epsilon}=constant$, which is facilitated by the nonconservative pressure formulation in \eqref{eq:system_multilayer}.
For the SWE this allows us to design a reconstruction that preserves the total water height, which will be a key to entropy conservation in Section \ref{sec:entropy_stability}
\begin{lemma}[Well-balancedness]\label{lemma:wb}
	The hydrostatic reconstruction scheme \eqref{eq:hydrostatic_reconstruction_scheme} applied to the SWE \eqref{eq:system_swe} preserves the lake-at-rest initial condition \eqref{eq:lake-at-rest_swe}.
\end{lemma}	
\begin{proof}
	For well-balancedness we need to ensure that both conservative and nonconservative interface contributions vanish under lake-at-rest condition
	\begin{equation}
		\oint_{\partial E} ({\statevec{f}}_{\epsilon}^* + ({\statevecGreek{\phi}}\circ 	{\statevec{r}})^{\diamond}_{\epsilon})\cdot \vec{n} \operatorname{dS} = 0.
	\end{equation} 
	Assuming that the numerical flux is consistent with the physical flux, $\statevec{f}^*_{\epsilon}=0$ follows directly from the lake-at-rest initial condition as $v=0$.
	For the nonconservative term we have to consider separately wet and dry interfaces as shown in Figure~\ref{fig:layer_configurations}a and \ref{fig:layer_configurations}b, respectively. 
	At fully wet interfaces the lake-at-rest condition \eqref{eq:lake-at-rest_swe} specifies a constant water height and the nonconservative term vanishes since $\statevec{r}_{\epsilon}^{\pm} = (0, H^\pm_{\epsilon})^T$ is constant, whereas at dry interfaces the reconstruction satisfies $\statevecGreek{\phi}^{\pm}_{\epsilon} = (h_{\epsilon}^{\pm}, h_{\epsilon}^{\pm})^T = 0$ and the scheme is well-balanced.
	\qed 
\end{proof}

Now that well-balancedness for the SWE is established, we investigate the ML-SWE, where well-balancedness becomes a more subtle issue.
The standard hydrostatic reconstruction procedures for the SWE rely on the fact that bottom topography and water height are interchangeable with regards to the hydrostatic pressure, see e.g. \cite{audusse2004fast}.
However, this property does not carry over to the ML-SWE due to the additional density quotient $\sigma$ in the nonconservative term.
As a consequence, it seems that a single bottom reconstruction cannot satisfy the pressure balance in different layers simultaneously. 
This fact has lead many authors to design independent bottom or interface reconstructions to achieve well-balancedness, see e.g. \cite{martinez20201d, dudzinski2013well, castro2005numerical}. While these approaches successfully obtain well-balancedness, a physical interpretation is difficult and the implications, specifically on entropy conservation, were not investigated.
Using the new hydrostatic reconstruction we will demonstrate that well-balancedness and entropy stability for the ML-SWE can be achieved with a single bottom topography. It is, again, key that using the nonconservative pressure formulation either $h_m=0$ or $\statevec{r}_{m}=constant$ can lead to well-balancedness.
\begin{lemma}[Well-balancedness ML-SWE]\label{lemma:wb_ml}
	The hydrostatic reconstruction scheme \eqref{eq:hydrostatic_reconstruction_scheme} applied to the ML-SWE \eqref{eq:system_multilayer} preserves the lake-at-rest initial condition \eqref{eq:lake-at-rest_multilayer}.
\end{lemma}
\begin{proof}
	Analogous to the SWE, for a consistent numerical flux we obtain $\statevec{f}^*_{\epsilon}=0$ from the lake-at-rest condition \eqref{eq:lake-at-rest_multilayer} as $v=0$.
	However, the nonconservative contribution now introduces additional difficulties due to the density quotients $\sigma$. To show well-balancedness we follow a similar reasoning as in Lemma 1 and consider two possible configurations for each layer, as depicted in Figure~\ref{fig:layer_configurations}.
	
	If a layer $m$ is dry, the reconstruction procedure together with the lake-at-rest condition \eqref{eq:lake-at-rest_multilayer} ensures that $h_{m,\epsilon}^{\pm}=0$ and hence $\statevecGreek{\phi}^{\pm}_m=0$.
	On the other hand, if a layer $m$ is wet, the total layer height $H^{\pm}_{m} = \sum_{k\geq m} h^{\pm}_{m,\epsilon}+b^{\pm}_{\epsilon}$ is constant according to the lake-at-rest condition \eqref{eq:lake-at-rest_multilayer}. Layers above are not affected by the reconstruction such that $h_{k,\epsilon}$ remain constant for $k<m$.	
	So, in each layer we obtain either $\statevecGreek{\phi}^{\pm}_m=0$ or $\statevec{r}_{m,\epsilon}=constant$ such that the nonconservative term vanishes and the steady state is preserved.\qed 
\end{proof}
\begin{remark}
	The well-balancedness is achieved solely by the reconstruction procedure and is satisfied for any consistent numerical and nonconservative flux.
\end{remark} 
\subsection{Entropy stability}\label{sec:entropy_stability}
We next demonstrate semi-discrete entropy stability of the new hydrostatic reconstruction scheme \eqref{eq:hydrostatic_reconstruction_scheme}. 
First, we examine the entropy properties and derive an entropy inequality in the continuous case.
From this we proceed to mimic the analysis in the semi-discrete case. Therefore, we first create an entropy conservative scheme, that is then complemented with suitable dissipation to obtain an entropy stable scheme.\\

The SWE are endowed with a natural convex entropy, entropy flux pair related to the total energy \cite{fjordholm2012energy}
\begin{equation}
	\begin{aligned}
		S(\statevec{u}) &= \frac{1}{2}hv^2 + \frac{1}{2}gh^2 + ghb,\\
		{f}^S(\statevec{u}) &= \left(\frac{hv^2}{2} + gh^2\right)v + ghbv,
	\end{aligned}
\end{equation}
which generalizes to the total energy within the layers for the ML-SWE \cite{bouchut2010robust}
\begin{equation}
	\begin{aligned}
		S(\statevec{u}) &= 
		\sum\limits_{m=1}^M \rho_m\left(\frac{h_mv_m^2}{2} 
		+ \frac{gh_m^2}{2}
		+ gh_mb 
		+ g\sum_{k<m}\sigma_{km}h_kh_m
		\right),
		\\
		{f}^S(\statevec{u}) &= 
		\sum\limits_{m=1}^M \rho_mv_m\left(\frac{h_mv_m^2}{2} 
		+ gh_mb
		+ g\sum_{k\geq m} h_kh_m
		+ g\sum_{k<m}\sigma_{km}h_kh_m
		\right).
	\end{aligned}
\end{equation}

Differentiation of the entropy function $S(\statevec{u})$ with respect to the conservative variables yields the entropy variables $\statevec{w} = (\statevec{w}^T_1,...,\statevec{w}^T_M) = \partial_{\statevec{u}}S(\statevec{u})$

\begin{equation}
	\statevec{w}_m =
	\begin{pmatrix}
		\rho_mg\left(b + \sum\limits_{k\geq m}h_k + \sum\limits_{k<m} \sigma_{km}h_k\right) - \frac{1}{2}\rho_mv_m^2\\
		\rho_mv_m
	\end{pmatrix}, \quad m=1,...,M.
\end{equation}
We contract \eqref{eq:balance_law_continuous} with the entropy variables to obtain the scalar entropy conservation law
\begin{equation}
	\statevec{w}^T\statevec{u}_t + \statevec{w}^T\left(\partial_x\statevec{f} + \statevecGreek{\phi}\circ\left(\partial_x\statevec{r}\,\right)\right) 
	= S_t + \partial_x{f}^{\,S}
	= 0,
	\label{eq:entropy_conservation_law}
\end{equation}
and integrate over the domain $\Omega$ and apply Gauss' Law to obtain entropy conservation in integral form
\begin{equation}
	\int\limits_{\Omega}S_t \operatorname{dV} + \int\limits_{\partial \Omega}{f}^{\,S}\cdot\spacevec{n} \dS = 0,
	\label{eq:entropy_conservation_law_integral_form}
\end{equation}
which states that for smooth solutions the total entropy is conserved. However, in the presence of discontinuities, entropy must be dissipated which leads to the entropy inequality
\begin{equation}
	\int\limits_{\Omega}S_t \operatorname{dV} + \int\limits_{\partial \Omega}{f}^{\,S} \cdot \spacevec{n} \dS \leq 0.
	\label{eq:entropy_inequality}
\end{equation}
Using a nonconservative formulation of the pressure term in \eqref{eq:system_swe} and \eqref{eq:system_multilayer}, the advective flux contracts to the entropy flux, and simplifies some proofs in the later part of this section.
\begin{equation}
	\statevec{w}^T\statevec{f} = {f}^{\,S}.
	\label{eq:entropy_flux_potential}
\end{equation}

To demonstrate entropy stability of the method, we mimic the continuous entropy analysis and demonstrate that we recover the integral entropy conservation law \eqref{eq:entropy_conservation_law_integral_form} and entropy inequality \eqref{eq:entropy_inequality} in the semi-discrete setting.\\

Therefore, we consider a single interface with adjacent FV cells of size $\Delta x$ that are denoted by superscripts ``$\boldsymbol{-}$'' and ``$\boldsymbol{+}$''. On this domain we then formulate the hydrostatic reconstruction scheme \eqref{eq:hydrostatic_reconstruction_scheme} with nonconservative term \eqref{eq:nonconservative_flux} in each FV cell to obtain
\begin{equation}
	\Delta x \frac{\partial}{\partial 	t}\statevec{u}^- 
	=
	\statevec{f}^- - \statevec{f}^*_{\epsilon} - \frac{\statevecGreek{\phi}^-_{\epsilon}}{2}\circ\jump{\statevec{r}_{\epsilon}},
	\qquad
	\Delta x \frac{\partial}{\partial 	t}\statevec{u}^+ 
	= 
	\statevec{f}^*_{\epsilon}  - \statevec{f}^+ -  \frac{\statevecGreek{\phi}^+_{\epsilon}}{2}\circ\jump{\statevec{r}_{\epsilon}}.
\end{equation}

We contract with entropy variables and assume time continuity to recover the rate of entropy change in each cell
\begin{equation}
	\Delta x \frac{\partial}{\partial 	t}S^- 
	=
	(\statevec{w}^-)^T\left(
	\statevec{f}^- - \statevec{f}^*_{\epsilon} - \frac{\statevecGreek{\phi}^-_{\epsilon}}{2}\circ\jump{\statevec{r}_{\epsilon}}\right),
	\qquad
	\Delta x \frac{\partial}{\partial 	t}S^+ 
	= 
	(\statevec{w}^+)^T\left(\statevec{f}^*_{\epsilon}  - \statevec{f}^+ -  \frac{\statevecGreek{\phi}^+_{\epsilon}}{2}\circ\jump{\statevec{r}_{\epsilon}}\right).
\end{equation}
Summing over both cells and gathering terms we find the total entropy change
\begin{equation}
	\Delta x \frac{\partial}{\partial t}\left(S^-+S^+\right) 
	=
	\jump{\statevec{w}}^T\statevec{f}_{\epsilon}^* - \jump{\statevec{w}^T\statevec{f}} - \avg{\statevec{w}\circ\statevecGreek{\phi_{\epsilon}}}^T\jump{\statevec{r_{\epsilon}}}.
	\label{eq:fv_total_discrete_entropy}
\end{equation}
To recover a discrete version of the integral entropy conservation law \eqref{eq:entropy_conservation_law_integral_form} we require that

\begin{equation}
	\jump{\statevec{w}}^T\statevec{f}_{\epsilon}^* - \jump{\statevec{w}^T\statevec{f}} - \avg{\statevec{w}\circ\statevecGreek{\phi_{\epsilon}}}^T\jump{\statevec{r_{\epsilon}}}
	\overset{!}{=}
	-\jump{\spacevec{f^S}}.
\end{equation}
Substituting \eqref{eq:entropy_flux_potential} yields the following entropy conservation condition for the numerical flux $\statevec{f}^*_{\epsilon}$
\begin{equation}
	\jump{\statevec{w}}^T\statevec{f}^{EC}_{\epsilon} - \avg{\statevec{w}\circ \statevecGreek{\phi}_{\epsilon}}^T\jump{\statevec{r}_{\epsilon}} = 0.
	\label{eq:entropy_conservation_condition}
\end{equation}
where we introduce the notation $\statevec{f}^{EC}_{\epsilon}$ to denote an entropy conservative (EC) numerical flux, satisfying condition \eqref{eq:entropy_conservation_condition}. We demonstrate in the following that \eqref{eq:entropy_conservation_condition} is satisfied by the EC flux
\begin{equation}
	\statevec{f}_{m,\epsilon}^{EC} \coloneqq \begin{pmatrix}
		\avg{h_{m,\epsilon}{v_m}} \\
		\avg{h_{m,\epsilon}v_m}\avg{v_m}
	\end{pmatrix}.
	\label{eq:ec_flux}
\end{equation}

\begin{lemma}[Entropy conservation (SWE)]
	The hydrostatic reconstruction scheme \eqref{eq:hydrostatic_reconstruction_scheme} for the SWE \eqref{eq:system_swe} with the EC flux \eqref{eq:ec_flux} and nonconservative term \eqref{eq:nonconservative_flux} is entropy conservative.
	\label{lemma:ec_swe}
\end{lemma}
\begin{proof}
	For entropy conservation we need to show that
	\begin{equation}
		\begin{aligned}
			\jump{\statevec{w}}^T\statevec{f}^{EC}_{\epsilon} - \avg{\statevec{w}\circ \statevecGreek{\phi}_{\epsilon}}^T\jump{\statevec{r}_{\epsilon}}
			=
			\left(g\jump{h + b} - \frac{\jump{v^2}}{2}\right) \avg{h_{\epsilon}v}
			+ \jump{v}\avg{h_{\epsilon}v}\avg{v}
			- g\jump{h_{\epsilon} + b_{\epsilon}}\avg{h_{\epsilon}v} = 0.
		\end{aligned}
	\end{equation}
	We first use the relation $\frac{\jump{v^2}}{2} = \jump{v}\avg{v}$ to cancel some terms and obtain
	\begin{equation}
		\begin{aligned}
			\jump{\statevec{w}}^T\statevec{f}^{EC}_{\epsilon} - 	\avg{\statevec{w}\circ \statevecGreek{\phi}_{\epsilon}}^T\jump{\statevec{r}_{\epsilon}}
			=
			g\jump{h + b} \avg{h_{\epsilon}v}
			- g\jump{h_{\epsilon} + b_{\epsilon}}\avg{h_{\epsilon}v}
			=
			0,
		\end{aligned}
	\end{equation}
	where entropy conservation follows directly as the reconstruction preserves the total water height, i.e., $h+b = h_{\epsilon}+b_{\epsilon}$. \qed
\end{proof}

Next, we consider the ML-SWE.
In this case, due to differing layer densities, the bottom reconstruction causes entropy violations, even for the EC flux.
While EC cannot be achieved, we are able to quantify the amount of entropy violation, which provides the necessary information to design entropy stable fluxes.

\begin{lemma}[Entropy conservation (ML-SWE)]
	The hydrostatic reconstruction scheme \eqref{eq:hydrostatic_reconstruction_scheme} for the ML-SWE with the EC flux \eqref{eq:ec_flux} and nonconservative term \eqref{eq:nonconservative_flux} generates the entropy contribution 
	\begin{equation}
		\jump{\statevec{w}}^T\statevec{f}^{EC}_{\epsilon} - \avg{\statevec{w}\circ \statevecGreek{\phi}_{\epsilon}}^T\jump{\statevec{r}_{\epsilon}}
		=
		g\sum\limits_{m=1}^M \rho_m 
		\sum_{k<m}(\sigma_{km}-1)\jump{h_k - h_{k,\epsilon}}\avg{h_{m,\epsilon}v_m}
	\end{equation}
	and therefore violates entropy conservation.
	\label{lemma:entropy_contribution_MLSWE}
\end{lemma}
\begin{proof}
	The proof is similar to the previous one for the SWE. However, due to the additional coupling terms proportional to $\sigma_{km}$ in each layer $w=w_{\epsilon}$ no longer holds. It is precisely these coupling terms that generate the entropy contributions. As before we expand the terms from \eqref{eq:entropy_conservation_condition} to have
	\begin{equation}
		\begin{aligned}
			&\jump{\statevec{w}}^T\statevec{f}^{EC}_{\epsilon} - \avg{\statevec{w}\circ \statevecGreek{\phi}_{\epsilon}}^T\jump{\statevec{r}_{\epsilon}}
			\\&\quad=
			\sum\limits_{m=1}^M\rho_m\left(g
			\jump{\sum_{k\geq m}h_k + \sum_{k<m}\sigma_{km}h_k + b}
			-\frac{\jump{v_m^2}}{2}\right)\avg{h_{m,\epsilon}v_m}
			+ \rho_m\jump{v_m}\avg{h_{m,\epsilon}v_m}\avg{v_m}
			\\&\qquad-
			\rho_m g\jump{\sum_{k\geq m}h_{k,\epsilon} + \sum_{k<m}\sigma_{km}h_{k,\epsilon} + b_{\epsilon}}\avg{h_{m,\epsilon}v_m}.
		\end{aligned}
	\end{equation}
	Again we use $\frac{\jump{v_m^2}}{2} = \jump{v_m}\avg{v_m}$ to cancel some terms and obtain
	\begin{equation}
		\begin{aligned}
			&\jump{\statevec{w}}^T\statevec{f}^{EC}_{\epsilon} - \avg{\statevec{w}\circ \statevecGreek{\phi}_{\epsilon}}^T\jump{\statevec{r}_{\epsilon}}
			\\& \quad=
			g\sum\limits_{m=1}^M \rho_m 
			\jump{\sum_{k\geq m}h_k + \sum_{k<m}\sigma_{km}h_k + b}
			\avg{h_{m,\epsilon}v_m}
			-
			\rho_m \jump{\sum_{k\geq m}h_{k,\epsilon} + \sum_{k<m}\sigma_{km}h_{k,\epsilon} + b_{\epsilon}}\avg{h_{m,\epsilon}v_m}
		\end{aligned}
		\label{eq:step_wb_proof_ml_swe}
	\end{equation}
	
	We add and subtract $\sum_{k<m}h_k$ within the first jump term and use that the reconstruction preserves the total layer height $H_1$ such that $\jump{\sum_{k=1}^M h_k + b} = \jump{\sum_{k=1}^M h_{k,\epsilon} + b_{\epsilon}}$ to obtain the relation
	\begin{equation}
		\jump{\sum_{k\geq m}h_k + \sum_{k<m}\sigma_{km}h_k + b} = \jump{\sum_{k=1}^M h_{k,\epsilon} + b_{\epsilon}} + \sum_{k<m}(\sigma_{km}-1)\jump{h_k}.
	\end{equation}
	
	Substituting this relation into \eqref{eq:step_wb_proof_ml_swe} we find
	\begin{equation}
		\begin{aligned}
			&\jump{\statevec{w}}^T\statevec{f}^{EC}_{\epsilon} - \avg{\statevec{w}\circ \statevecGreek{\phi}_{\epsilon}}^T\jump{\statevec{r}_{\epsilon}}
			\\&\quad=
			g\sum\limits_{m=1}^M \rho_m 
			\jump{\sum_{k=1}^M h_{k,\epsilon} - \sum_{k\geq m}h_{k,\epsilon} - \sum_{k<m}\sigma_{km}h_{k,\epsilon}}\avg{h_{m,\epsilon}v_m}
			+
			\rho_m \sum_{k<m}(\sigma_{km}-1)\jump{h_k}\avg{h_{m,\epsilon}v_m}
			\\&\quad=
			g\sum\limits_{m=1}^M \rho_m 
			\jump{\sum_{k<m}h_{k,\epsilon} - \sum_{k<m}\sigma_{km}h_{k,\epsilon}}\avg{h_{m,\epsilon}v_m}
			+
			\rho_m \sum_{k<m}(\sigma_{km}-1)\jump{h_k}\avg{h_{m,\epsilon}v_m}
			\\&\quad=
			g\sum\limits_{m=1}^M \rho_m 
			\sum_{k<m}(\sigma_{km}-1)\jump{h_k - h_{k,\epsilon}}\avg{h_{m,\epsilon}v_m},
		\end{aligned}
		\label{eq:proof_ec_mlswe_entropy_contribution}
	\end{equation}
	which shows that the hydrostatic reconstruction for ML-SWE is not entropy conservative. \qed
\end{proof}
\begin{remark}
	From \eqref{eq:proof_ec_mlswe_entropy_contribution} it follows that the entropy violation is generated due to the hydrostatic reconstruction procedure. Without the hydrostatic reconstruction the flux $\statevec{f}^{EC}$ is actually entropy conservative as $\jump{h_k - h_{k,\epsilon}} = 0$.
\end{remark}
Even though $\statevec{f}^{EC}_{\epsilon}$ for the ML-SWE violates entropy conservation, we keep the notation as the flux is indeed EC if applied without the hydrostatic reconstruction.

As solutions of \eqref{eq:balance_law_continuous} may develop discontinuities in finite time, the EC scheme is only an intermediate step. Instead our objective is to recover an entropy stable scheme that satisfies the entropy inequality \eqref{eq:entropy_inequality} and ensures entropy dissipation at shocks. From this requirement, analogous to the entropy conservation condition \eqref{eq:entropy_conservation_condition}, we derive the following condition for an entropy stable (ES) numerical flux
\begin{equation}
	\jump{\statevec{w}}^T\statevec{f}^{ES}_{\epsilon} - \avg{\statevec{w}\circ \statevecGreek{\phi}_{\epsilon}}^T\jump{\statevec{r}_{\epsilon}} \leq 0.
\end{equation}

Following the idea of Tadmor \cite{tadmor2003entropy}, we aim construct such an entropy stable flux $\statevec{f}^{ES}$ by adding local Lax-Friedrichs (LLF) type numerical dissipation to our entropy conserving baseline scheme
\begin{equation}
	\statevec{f}_{\epsilon}^{ES} \coloneqq \statevec{f}_{\epsilon}^{EC} - \frac{1}{2}\lambda_{\max}\jump{\statevec{u}_{\epsilon}},
	\label{eq:definition_es_flux}
\end{equation}
where $\lambda_{\max}$ denotes a local estimate for the fastest wave speed and $\statevec{u}_{\epsilon}$ are the reconstructed conservative variables \eqref{eq:reconstruction_conservative_variables}. Since there is no closed-form solution of the eigenvalues for layers $M \geq 3$, we consider the approximation from \cite{lastra2018efficient} to obtain the estimate
\begin{equation}
	\lambda_{\max} \coloneqq \max\left(\left|\frac{\sum_{m=1}^M hv_m^+}{\sum_{m=1}^M h_m^+}\right|, \left|\frac{\sum_{m=1}^M hv_m^-}{\sum_{m=1}^M h_m^-}\right|,
	\underset{m}{\max} \left| \,v^+_m\right|,
	\underset{m}{\max} \left| \,v^-_m\right| \right) 
	+ 
	\max\left(\sqrt{g\sum_{m=1}^M h_m^+}, \sqrt{g\sum_{m=1}^M h_m^-}\right).
	\label{eq:lambda_max}
\end{equation} 

For entropy dissipation it is then typically required to make a change of variables $\jump{\statevec{u}} \approxeq \statevec{H}\jump{\statevec{w}}$, where $\statevec{H}$ is the symmetric positive definite Jacobian of the entropy variables \cite{carpenter2014entropy}.
However, we found that this approach faces several problems when applied to the present scheme.
First, as the entropy variables additionally depend on the bottom topography, in case of discontinuous bottom topographies there is no one-to-one mapping between conservative and entropy variables. This leads to the loss of the well-balanced property at wet/dry transitions, whereas the formulation in terms of reconstructed conservative variables from \eqref{eq:definition_es_flux} preserves the well-balanced property.
\begin{lemma}
	The entropy stable flux \eqref{eq:definition_es_flux} preserves the lake-at-rest conditions (\ref{eq:lake-at-rest_swe},\ref{eq:lake-at-rest_multilayer}), if the underlying EC flux is well-balanced.
\end{lemma}
\begin{proof}
	Under the lake-at-rest conditions \eqref{eq:lake-at-rest_swe} and \eqref{eq:lake-at-rest_multilayer} the reconstructed conservative variables \eqref{eq:reconstruction_conservative_variables} satisfy $\jump{\statevec{u}_{\epsilon}}=0$ such that the dissipation term in \eqref{eq:definition_es_flux} vanishes and well-balanced follows from $\statevec{f}_{\epsilon}^{ES} = \statevec{f}_{\epsilon}^{EC}$. \qed
\end{proof}

In addition the formulation in terms of entropy variables is numerically unstable for the ML-SWE as the Jacobian $\statevec{H}$ may become singular for vanishing layer heights.

Considering these issues, we now show that for the present systems entropy stability can be achieved with standard LLF dissipation making the change of variables $\jump{\statevec{u}}=\statevec{H}\jump{\statevec{w}}$ obsolete.
This provides a remedy for the previous issues and is computationally more efficient as it avoids the additional matrix-vector multiplication.

\begin{lemma}[Entropy stability for the SWE]
	The hydrostatic reconstruction scheme \eqref{eq:hydrostatic_reconstruction_scheme} for the SWE \eqref{eq:system_swe} with ES flux \eqref{eq:definition_es_flux} and nonconservative term \eqref{eq:nonconservative_flux} is entropy stable.
\end{lemma}
\begin{proof}
	We know from Lemma \ref{lemma:ec_swe}, that $\statevec{f}^{EC}$ leads to entropy conservation. Entropy stability then follows if the ES flux is more dissipative than the EC flux, so we must show that
	\begin{equation}
		\jump{\statevec{w}}^T\left(-\frac{1}{2}\lambda_{max}\jump{\statevec{u}_{\epsilon}}\right)= -\frac{1}{2}\lambda_{max} \jump{\statevec{w}}^T\jump{\statevec{u}_{\epsilon}}
		\leq 0,
	\end{equation}
	which is satisfied provided $\jump{\statevec{w}}^T\jump{\statevec{u}_{\epsilon}} \geq 0$.
	
	For the SWE we then use that the reconstruction satisfies $\statevec{w} = \statevec{w}_{\epsilon}$. Now expanding the variables we obtain
	\begin{equation}
		\jump{\statevec{w}}^T\jump{\statevec{u}_{\epsilon}}
		=
		\jump{\statevec{w_{\epsilon}}}^T\jump{\statevec{u}_{\epsilon}} 
		=
		\left(g\jump{h_{\epsilon}+b_{\epsilon}}-\frac{\jump{v^2}}{2}\right) \jump{h_{\epsilon}}
		+
		\jump{v}\jump{h_{\epsilon}v}.
		\label{eq:proof_es_swe_step_1}
	\end{equation}
	Combining the last two terms we find the equality
	\begin{equation}
		-\frac{\jump{v^2}}{2}\jump{h_{\epsilon}}
		+
		\jump{v}\jump{h_{\epsilon}v}
		= \avg{h_{\epsilon}}\jump{v}^2,
		\label{eq:proof_es_swe_product_rule}
	\end{equation}
	which is a discrete analogue of the product rule
	\begin{equation}
		-\frac{1}{2}v_x^2h_x + v_x(hv)_x
		=
		-h_xvv_x + h_xvv_x + hv_xv_x  = hv_x^2,
	\end{equation}
	where subscript $x$ denotes a derivative.
	Substituting \eqref{eq:proof_es_swe_product_rule} into \eqref{eq:proof_es_swe_step_1} we obtain quadratic terms and use $\avg{h_{\epsilon}}\geq0$, to obtain
	\begin{equation}
		g\jump{h_{\epsilon}+b_{\epsilon}}\jump{h_{\epsilon}} + \avg{h_{\epsilon}}\jump{v}^2
		=
		g\jump{b_{\epsilon}}\jump{h_{\epsilon}} + 	g\jump{h_{\epsilon}}^2+ \avg{h_{\epsilon}}\jump{v}^2
		\geq 
		g\jump{b_{\epsilon}}\jump{h_{\epsilon}}.
	\end{equation}
	From the novel reconstruction it follows that $g\jump{b_{\epsilon}}\jump{h_{\epsilon}} \geq 0 $. 
	As depicted in Figure~\ref{fig:layer_configurations}, we may encounter either wet, dry, or partially dry interfaces.
	At fully wet interfaces the bottom reconstruction is continuous and the contribution vanishes from $\jump{b_{\epsilon}}=0$, whereas on dry interfaces we obtain $\jump{h_{\epsilon}} = 0$ as the water heights are zero and again the contribution vanishes.
	At partially dry interfaces we obtain $\text{sgn}(\jump{b_{\epsilon}}) = \text{sgn}(\jump{h_\epsilon})$ and the contribution is positive, which shows that $g\jump{b_{\epsilon}}\jump{h_{\epsilon}} \geq 0 $ and accordingly $\jump{\statevec{w}}^T\jump{\statevec{u}_{\epsilon}} \geq 0$.
	\qed
\end{proof}

For the ML-SWE we demonstrated in Lemma \ref{lemma:entropy_contribution_MLSWE} that the EC flux introduces an additional entropy contribution. Even though the EC flux leads to entropy violation, entropy stability can be achieved if the entropy production is countered by sufficient entropy dissipation of the form introduced in \eqref{eq:definition_es_flux}.

\begin{lemma}[Entropy stability for the ML-SWE]
	The hydrostatic reconstruction scheme \eqref{eq:hydrostatic_reconstruction_scheme} for the ML-SWE \eqref{eq:system_multilayer} with ES flux \eqref{eq:definition_es_flux} and nonconservative term \eqref{eq:nonconservative_flux} is entropy stable.
\end{lemma}
\begin{proof}
	For entropy stability we must show that 
	\begin{equation}
		\jump{\statevec{w}}^T\statevec{f}^{ES}_{\epsilon} - \avg{\statevec{w}\circ \statevecGreek{\phi}_{\epsilon}}^T\jump{\statevec{r}_{\epsilon}} 
		\leq 0,
	\end{equation}
	or equivalently
	\begin{equation}
		\jump{\statevec{w}}^T\statevec{f}^{EC}_{\epsilon} - \avg{\statevec{w}\circ \statevecGreek{\phi}_{\epsilon}}^T\jump{\statevec{r}_{\epsilon}} 
		- \frac{\lambda_{\max}}{2}\jump{\statevec{w}}^T\jump{\statevec{u}_{\epsilon}} 
		\leq 0.
	\end{equation}
	From Lemma \ref{lemma:entropy_contribution_MLSWE} we know that the EC scheme generates the additional entropy contribution
	\begin{equation}
		\jump{\statevec{w}}^T\statevec{f}^{EC}_{\epsilon} - \avg{\statevec{w}\circ \statevecGreek{\phi}_{\epsilon}}^T\jump{\statevec{r}_{\epsilon}} 
		= 
		g\sum\limits_{m=1}^M \rho_m  
		\sum_{k<m}(\sigma_{km}-1)\jump{h_k - h_{k,\epsilon}}\avg{h_{m,\epsilon}v_m}.
	\end{equation}
	
	In order to achieve entropy stability, we ensure that the additional dissipation in the ES flux counters this entropy production. Therefore, we evaluate the additional entropy dissipation introduced by the ES flux and demonstrate that it indeed counters the entropy production such that
	\begin{equation}
		g\sum\limits_{m=1}^M \rho_m 
		\sum_{k<m}(\sigma_{km}-1)\jump{h_k - h_{k,\epsilon}}\avg{h_{m,\epsilon}v_m}
		-\frac{\lambda_{\max}}{2}\jump{\statevec{w}}^T\jump{\statevec{u}_{\epsilon}}
		\leq 0.
		\label{eq:es_proof_proposition_1}
	\end{equation}
	To examine the contribution of the dissipation we expand the inner product in the last term to find
	\begin{equation}
		\jump{\statevec{w}}^T\jump{\statevec{u}_{\epsilon}}
		=
		\sum\limits_{m=1}^M \rho_m \left(g
		\jump{\sum_{k\geq m}h_k + \sum_{k<m}\sigma_{km}h_k + b}
		-\frac{\jump{v_m^2}}{2}\right)\jump{h_{m,\epsilon}}
		+ \rho_m \jump{v_m}\jump{h_{m,\epsilon}v_m}.
	\end{equation}
	
	Analogous to the SWE we use the discrete product rule to combine terms
	\begin{equation}
		\jump{v_m}\jump{h_{m,\epsilon}v_m} - \frac{\jump{v_m^2}}{2}\jump{h_{m,\epsilon}} 
		= \avg{h_{m,\epsilon}}\jump{v_m}^2
		\label{eq:velocity_quadratic}
	\end{equation}
	and rewrite the velocity contributions into a quadratic form. This yields the inequality
	\begin{equation}
		\begin{aligned}
			\jump{\statevec{w}}^T\jump{\statevec{u}_{\epsilon}}
			&=
			\sum\limits_{m=1}^M
			\rho_m g\jump{\sum_{k\geq m}h_k + \sum_{k<m}\sigma_{km}h_k + b_{\epsilon}}\jump{h_{m,\epsilon}} + \rho_m \avg{h_{m,\epsilon}}\jump{v_m}^2
			\\ &\geq
			g\sum\limits_{m=1}^M \rho_m
			\jump{\sum_{k\geq m}h_k + \sum_{k<m}\sigma_{km}h_k + b}\jump{h_{m,\epsilon}}. 
		\end{aligned}
		\label{eq:contraction_in_lemma8}
	\end{equation}
	
	Next, we add and subtract $\sum_{k<m}h_k$ within the jump term and use that the reconstruction preserves the total layer height $H_1$ to find 
	\begin{equation}
		\begin{aligned}
			\jump{\statevec{w}}^T\jump{\statevec{u}_{\epsilon}}
			&\geq
			g\sum\limits_{m=1}^M \rho_m 
			\jump{\sum_{k=1}^M h_k + b}\jump{h_{m,\epsilon}} 
			+ 
			\rho_m \sum_{k<m}\left(\sigma_{km}-1\right)\jump{h_k}\jump{h_{m,\epsilon}}
			\\&=
			g\sum\limits_{m=1}^M \rho_m 
			\sum\limits_{k=1}^M \jump{h_{k,\epsilon}}\jump{h_{m,\epsilon}} 
			+ g\rho_m \jump{b_{\epsilon}}\jump{h_{m,\epsilon}} 
			+ \rho_m \sum_{k<m}\left(\sigma_{km}-1\right)\jump{h_k}\jump{h_{m,\epsilon}}
			\\&\geq
			g\sum\limits_{m=1}^M \rho_m 
			\sum\limits_{k=1}^M \jump{h_{k,\epsilon}}\jump{h_{m,\epsilon}} 
			+ \rho_m \sum_{k<m}\left(\sigma_{km}-1\right)\jump{h_k}\jump{h_{m,\epsilon}}.
		\end{aligned}
		\label{eq:proof_es_mlswe_step_1}
	\end{equation}
	The last inequality follows, again, from the reconstruction procedure. Analogous to the SWE the reconstruction ensures that either $\jump{b_{\epsilon}} = 0$ in the wet case, $\jump{h_{k,\epsilon}}=0$ in the dry case, or $\text{sgn}(\jump{b_{\epsilon}}) = \text{sgn}(\jump{h_{m,\epsilon}})$ in the partially dry case. Therefore, $g\jump{b_{\epsilon}}\jump{h_{m,\epsilon}}$ is always non-negative.
	
	We further manipulate \eqref{eq:proof_es_mlswe_step_1} where we add and subtract $h_{k,\epsilon}$ to derive a quadratic form with a MIN matrix that is symmetric positive semidefinite, as shown in \cite[Theorem 8.1]{mattila2016studying}.
	\begin{equation}
		\begin{aligned}
			\jump{\statevec{w}}^T\jump{\statevec{u}_{\epsilon}}
			&\geq
			g\sum\limits_{m=1}^M \rho_m 
			\sum\limits_{k=1}^M \jump{h_{k,\epsilon}}\jump{h_{m,\epsilon}} 
			+ \rho_m \sum_{k<m}\left(\sigma_{km}-1\right)\jump{h_k}\jump{h_{m,\epsilon}}
			\\&=
			g\sum\limits_{m=1}^M \rho_m 
			\sum\limits_{k=1}^M \jump{h_{k,\epsilon}}\jump{h_{m,\epsilon}} 
			+ \rho_m \sum_{k<m}\left(\sigma_{km}-1\right)\jump{h_{k,\epsilon}}\jump{h_{m,\epsilon}}
			+ \rho_m \sum_{k<m}\left(\sigma_{km}-1\right)\jump{h_k - h_{k,\epsilon}}\jump{h_{m,\epsilon}}
			\\&=
			g\begin{pmatrix}
				\jump{h_{1,\epsilon}} \\ \vdots \\  \jump{h_{M,\epsilon}}
			\end{pmatrix}^T
			\begin{pmatrix}
				\rho_1  & \rho_1 & \hdots & \rho_1 \\
				\rho_1 & \rho_2 & \hdots & \rho_2  \\
				\vdots & \vdots & \ddots & \vdots  \\
				\rho_1 & \rho_2 & \hdots & \rho_M \\
			\end{pmatrix}
			\begin{pmatrix}
				\jump{h_{1,\epsilon}} \\ \vdots \\ \jump{h_{M,\epsilon}}
			\end{pmatrix}
			+ g\sum\limits_{m=1}^M \rho_m \sum_{k<m}\left(\sigma_{km}-1\right)\jump{h_k - h_{\epsilon,k}}\jump{h_{m,\epsilon}}
			\\&\geq
			g\sum\limits_{m=1}^M \rho_m \sum_{k<m}\left(\sigma_{km}-1\right)\jump{h_k - h_{\epsilon,k}}\jump{h_{m,\epsilon}}.
		\end{aligned}
	\end{equation} 
	We then combine this result with the entropy contribution from Lemma \ref{lemma:entropy_contribution_MLSWE} to obtain
	\begin{equation}
		\begin{aligned}
			&\jump{\statevec{w}}^T\statevec{f}^{ES}_{\epsilon} - \avg{\statevec{w}\circ \statevecGreek{\phi}_{\epsilon}}^T\jump{\statevec{r}_{\epsilon}}
			\leq
			g\sum\limits_{m=1}^M \rho_m 
			\sum_{k<m}(\sigma_{km}-1)\jump{h_k - h_{k,\epsilon}}
			\left(\avg{h_{m,\epsilon}v_m} - \jump{h_{m,\epsilon}}\frac{\lambda_{\max}}{2} \right)
		\end{aligned}
	\end{equation}
	Now if layers $k<m$ are not reconstructed then $\jump{h_k - h_{k,\epsilon}} = 0$ and entropy stability is achieved. So,
	we need only consider the case that there is some reconstruction in layers $k<m$.
	Assuming the reconstruction takes place on the $\boldsymbol{-}$ side, the $\boldsymbol{+}$ side is not reconstructed and $h_k^+ - h_{k,\epsilon}^+ = 0$. If any layer $k<m$ is reconstructed then, from linearity of the bottom reconstruction, it follows that layer $m$ on the respective side must be dry, such that $h_{m,\epsilon}^- = 0$ and
	\begin{equation}
		\begin{aligned}
			&\jump{\statevec{w}}^T\statevec{f}^{ES}_{\epsilon} - \avg{\statevec{w}\circ \statevecGreek{\phi}_{\epsilon}}^T\jump{\statevec{r}_{\epsilon}}
			\leq
			g\sum\limits_{m=1}^M \rho_m 
			\sum_{k<m}(\sigma_{km}-1)(h_{k,\epsilon}^- - h_k^-)
			\frac{h_{m}^+}{2}\left(v^+_m - \lambda_{\max}\right).
		\end{aligned}
	\end{equation}
	An analogous expression is obtained when considering reconstruction on the $\boldsymbol{+}$ element 
	and entropy stability follows from $\lambda_{\max} \geq \{v^+_m,v^-_m\}$, $0 \leq h^{\pm}_{k,\epsilon} \leq h_k^{\pm}$ and $\sigma_{km}<1$.
	\qed
\end{proof}

\subsection{Positivity-preservation}
Ensuring positivity-preservation is a key requirement when considering wetting and drying for shallow water flows to avoid unphysical solutions and related numerical problems. The following Lemma establishes positivity-preservation of the present hydrostatic reconstruction scheme under a given time step restriction.
\begin{lemma}[Positivity-preservation]
	The hydrostatic reconstruction scheme \eqref{eq:hydrostatic_reconstruction_scheme} with ES flux \eqref{eq:definition_es_flux} is positivity preserving under the CFL condition
	\begin{equation}
		\Delta t \leq \frac{\Delta x}{2\lambda_{max}}
	\end{equation}
	when a first-order forward Euler method for time integration is used.
\end{lemma}
\begin{proof}
	The hydrostatic reconstruction of the water height \eqref{eq:reconstruction_waterheight} ensures that $h_{m,\epsilon}\geq 0$.
	In the mass equations the ES flux then recovers a standard LLF scheme, which has been proven to be positivity preserving under the given CFL condition in \cite{dong2023surface}. \qed 
\end{proof}
\begin{remark}\label{remark:SSPRK}
	The positivity result carries over to higher-order time integration using strong-stability-preserving Runge-Kutta (SSPRK) methods, which are a convex combination of the forward Euler method \cite{gottlieb2009high}. 
\end{remark}

\section{High-order reconstruction method}\label{sec:DG_method}
In this section we extend the low-order hydrostatic reconstruction scheme to high-order and two-dimensional curvilinear meshes by means of an entropy stable nodal discontinuous Galerkin spectral element method for nonconservative systems as introduced in \cite{ersing2024entropy}. 
The method builds upon a standard nodal collocation DGSEM formulation \cite{hesthaven2007nodal,kopriva2009implementing}, with additional nonconservative terms that are approximated in a path-conservative way, similar to the methods developed in \cite{franquet2012runge, renac2019entropy}. 
To achieve entropy stability, the summation-by-parts property of the method is used to rewrite divergence terms into a flux differencing formulation, with entropy conserving two-point fluxes \cite{carpenter2014entropy, gassner2016split}.
This approach was successfully applied to design well-balanced and entropy stable DGSEM for the SWE \cite{wintermeyer2017entropy} and two-layer shallow water equations \cite{ersing2024entropy}.

We further extend the method from \cite{ersing2024entropy} with shock-capturing and a positivity-limiter and use finite volume subcells in combination with the novel hydrostatic reconstruction to handle wet/dry transitions. Finally, we will demonstrate that the resulting approximation is positivity preserving and remains entropy stable and well-balanced, even for wet/dry transitions. Within this section, we restrict the discussion to the more general case of the ML-SWE, noting that all results carry over to the SWE when $M=1$.

First, to extend the balance law \eqref{eq:balance_law_continuous} to two spatial dimensions, we introduce block-vector notation as defined in \cite{winters2021construction} and obtain the reformulated balance law
\begin{equation}
	\statevec{u}_t + \spacevec{\nabla}_x \cdot \bigstatevec{f}(\statevec{u}) + \statevecGreek{\phi}(\statevec{u})\circ\left(\spacevec{\nabla}_x\cdot\bigstatevec{r}(\statevec{u})\right) = 0,
	\label{eq:balance_law_continuous_2d}
\end{equation}
with block vectors containing fluxes in both spatial dimensions
\begin{equation}
	\bigstatevec{f} = 
	\begin{pmatrix} \statevec{f}^{1} \\ \statevec{f}^2 \end{pmatrix}
	\quad 
	\bigstatevec{r}  = 
	\begin{pmatrix} \statevec{r}^1 \\ \statevec{r}^2\end{pmatrix},
\end{equation}
where $\statevec{f}^i = ({\statevec{f}^i_1}^T, ..., {\statevec{f}^i_M}^T)^T$ and $\statevec{r}^i = ({\statevec{r}_1^i}^T, ..., {\statevec{r}_M^i}^T)^T$ for $i=1,2$ with the layer-wise terms given by
\begin{equation}
	\begin{aligned}
	&\statevec{f}^1_m = \begin{pmatrix} h_mv_m \\ h_mv_m^2 \\ h_mv_mw_m\end{pmatrix},
	\quad
	\statevec{r}^1_m = \begin{pmatrix} 0 \\ b + \sum\limits_{k\geq m}h_k + \sum\limits_{k<m}\sigma_{km}h_k \\ 0 \end{pmatrix},
	\quad
	\statevec{f}^2_m = \begin{pmatrix} h_mw_m \\ h_mv_mw_m \\ h_mw_m^2\end{pmatrix},
	\quad
	\statevec{r}^2_m = \begin{pmatrix} 0 \\ 0 \\ b + \sum\limits_{k\geq m}h_k + \sum\limits_{k<m}\sigma_{km}h_k \end{pmatrix}.
	\end{aligned}
	\label{eq:physical_fluxes_2d}
\end{equation}
We then formulate a split-form DGSEM with nonconservative terms as described in \cite{ersing2024entropy}. 
Therefore, the physical domain $\Omega$ is divided into $K$ non-overlapping unstructured quadrilateral elements, that are mapped from physical space $\spacevec{x} = (x,y)^T$ onto a reference element $E=[-1,1]^2$ in computational space $\spacevec{\xi}=(\xi,\eta)^T$.

We map the system of equations by transforming the divergence operator into computational space
\begin{equation}
	\spacevec{\nabla}_x \cdot \bigstatevec{f} = 
	\frac{1}{J} \spacevec{\nabla}_{\xi} \cdot \bigcontravec{f},
	\quad 
	\statevecGreek{\phi} \circ \left(\spacevec{\nabla}_{x} \cdot \bigstatevec{r}\,\right) =
	\frac{1}{J}
	\statevecGreek{\phi} \circ \left(\spacevec{\nabla}_{\xi} \cdot \bigcontravec{r}\,\right),
	\label{eq:transform_flux_div}
\end{equation}
with the Jacobian of the mapping $J$ and the contravariant fluxes and nonconservative terms 
\begin{equation}
	\bigcontravec{f} = 
		\begin{pmatrix}
			Ja_1^1\normalmatrix{I} & Ja_1^2\normalmatrix{I} \\
			Ja_2^1\normalmatrix{I} & Ja_2^2\normalmatrix{I}
		\end{pmatrix}^T
		\bigstatevec{f},
	\quad
	\bigcontravec{r} = 
		\begin{pmatrix}
			Ja_1^1\normalmatrix{I} & Ja_1^2\normalmatrix{I} \\
			Ja_2^1\normalmatrix{I} & Ja_2^2\normalmatrix{I}
		\end{pmatrix}^T
		\bigstatevec{r},
\end{equation}
that depend on contravariant basis vectors $\spacevec{a}^i,\; i=1,2$ as well as the identity matrix $\normalmatrix{I}$.
For free-stream preservation and entropy conservation it is essential that the mapping satisfies the metric identities in a discrete setting
\begin{equation}
	\frac{\partial}{\partial \xi}J\spacevec{a}^{\,1} + \frac{\partial}{\partial \eta} J\spacevec{a}^{\,2} = 0.
	\label{eq:metric_identity}
\end{equation}
This property is guaranteed for isoparametric boundaries \cite{kopriva2006metric} and is essential for the proofs of both entropy conservation and well-balancedness, presented in the latter parts of this section.

In each element we then approximate the solution
\begin{equation}
	\statevec{u} \approx \statevec{U} = \sum\limits_{i,j=0}^N \statevec{U}_{ij} \ell_i(\xi) \ell_j(\eta)
\end{equation}
with a local tensor-product basis using one-dimensional nodal Lagrange interpolating polynomials $\{\ell_k\}_{k=0}^N$ of degree $N$ on Legendre-Gauss-Lobatto (LGL) nodes $\{\xi\}_{k=0}^N$ and denote interpolated values at LGL nodes with capital letters.
Using the Lagrange basis functions, we define the derivative operator
\begin{equation}
	\mathcal{D}_{ij} \coloneqq \left.\frac{\partial\ell_j}{\partial\xi}\right|_{\xi=\xi_i}, \quad i,j=0,...,N,
	\label{eq:Differentiation_matrix}
\end{equation}
and apply LGL quadrature rules for integration with collocated interpolation and quadrature nodes, to obtain a diagonal mass matrix 
\begin{equation}  
	\mathcal{M} = \text{diag}(\omega_{\hspace{0.5pt}0},...,\omega_{N}),
	\label{eq:mass_matrix}
\end{equation}
where $\{\omega_i\}_{i=0}^N$ denote the corresponding LGL quadrature weights, see \cite{kopriva2009implementing} for details. 
This combination of interpolation and differentiation operators has the diagonal norm SBP property as shown in \cite{gassner2013skew}
\begin{equation}
	(\mathcal{M}\mathcal{D}) + (\mathcal{M}\mathcal{D})^T = \mathcal{Q} + \mathcal{Q}^T = \mathcal{B},
	\label{eq:SBP_integration_by_parts}
\end{equation}
with the SBP matrix $\mathcal{Q}$ and the boundary matrix $\mathcal{B}=\text{diag}(-1, 0,...,0,1)$, which is key to prove entropy conservation.

In order to derive a split-form DGSEM we then multiply \eqref{eq:balance_law_continuous_2d} with a test function $\statevecGreek{\varphi}$ and integrate over the domain $\Omega$ using LGL quadrature rules and the SBP property \eqref{eq:SBP_integration_by_parts}
\begin{equation}
	\begin{aligned}
		\iprodN{\mathbb{I}^N(J)\statevec{U}_t,\statevecGreek{\varphi}} + 
		&\int\limits_{\partial E,N} \statevecGreek{\varphi}^T \{ \statevec{F}_n^* - \statevec{F}_n\} \hat{s} \dS + \iprodN{\vec{\mathbb{D}}\cdot\bigcontravec{F}^{\#} ,\statevecGreek{\varphi}} 
		\\+
		&\int\limits_{\partial E,N} \statevecGreek{\varphi}^T \left(\statevecGreek{\Phi}\circ\bigstatevec{R}\right)^{\diamond}_n \hat{s} \,\text{d}S
		+
		\iprodN{\spacevec{\mathbb{D}}\cdot(\statevecGreek{\Phi}\circ\bigcontravec{R})^{\#},\statevecGreek{\varphi}} = 0.
	\end{aligned}
	\label{eq:split_formulation}
\end{equation}
Here $\mathbb{I}^N(J)$ denotes the interpolated Jacobian of the mapping, $\statevec{F}_n = \bigstatevec{F}\cdot\spacevec{n}$ and $\left(\statevec{\Phi} \circ \bigstatevec{R}\right)^{\diamond}_n = \left(\statevec{\Phi} \circ \bigstatevec{R}\right)^{\diamond}\cdot\spacevec{n}$ denote the surface normal fluxes and numerical nonconservative term, $\bigcontravec{F}^{\#}$ and $(\statevecGreek{\Phi}\circ\bigcontravec{R})^{\#}$ denote volume fluxes and $\hat{s}$ is the differential surface element.
Furthermore, we introduce the discrete inner product for functions $f$ and $g$
\begin{equation}
	\iprodN{f,g} = \sum\limits_{i,j=0}^N f_{ij}g_{ij}\omega_{ij},
\end{equation} \\
with quadrature weights $\omega_{ij} = \omega_{i}\omega_{j}$ and the split-form divergence operator $\spacevec{\mathbb{D}}$ (see e.g. \cite{winters2021construction, ersing2024entropy, gassner2016split} for complete details), which applied to the conservative fluxes and nonconservative terms yields 
\begin{equation}
	\begin{aligned}
		\vec{\mathbb{D}} \cdot  \bigcontravec{F}^{\#} 
		&=
		2\sum_{l=0}^{N} \mathcal{D}_{il} \left(\bigstatevec{F}^{\#}(\statevec{U}_{ij}, \statevec{U}_{lj}) \cdot \avg{J\vec{a}^{\,1}}_{(i,l)j} \right)
		 +
		2\sum_{l=0}^{N} \mathcal{D}_{jl}   \left(\bigstatevec{F}^{\#}(\statevec{U}_{ij}, \statevec{U}_{il}) \cdot \avg{J\vec{a}^{\,2}}_{i(j,l)} \right),
		\\
		\spacevec{\mathbb{D}}\cdot(\statevecGreek{\Phi}\circ\bigcontravec{R})^{\#} &= 
		2 \sum_{l=0}^{N} \mathcal{D}_{il}\left( (\statevecGreek{\Phi}\circ\bigstatevec{R})^{\#}(\statevec{U}_{ij},\statevec{U}_{lj}) \cdot \avg{J\vec{a}^{\,1}}_{(i,l)j} \right)
		+
		2 \sum_{l=0}^{N} \mathcal{D}_{jl}\left( (\statevecGreek{\Phi}\circ\bigstatevec{R})^{\#}(\statevec{U}_{ij},\statevec{U}_{il}) \cdot \avg{J\vec{a}^{\,2}}_{i(j,l)} \right),
		\label{eq:split_form_flux}
	\end{aligned}
\end{equation}
where the two-point volume fluxes $\bigstatevec{F}^{\#}$ and nonconservative terms $(\statevecGreek{\Phi}\circ\bigstatevec{R})^{\#}$ are evaluated at arithmetic mean values and jumps
\begin{equation}
	\begin{aligned}
	\avg{\cdot}_{(i,l)j} &= \frac{1}{2} \left(\left(\cdot\right)_{ij} + (\cdot)_{lj}\right), &\quad \avg{\cdot}_{i\left(j,l\right)} &= \frac{1}{2} \left(\left(\cdot\right)_{ij} + \left(\cdot\right)_{il}\right),
	\\
	\jump{\cdot}_{(i,l)j} &= \left(\left(\cdot\right)_{lj} - (\cdot)_{ij}\right), &\quad \jump{\cdot}_{i\left(j,l\right)} &= \left(\left(\cdot\right)_{il} - \left(\cdot\right)_{ij}\right).
	\end{aligned}
\end{equation}

The split-form DGSEM offers the flexibility of choosing different fluxes for the volume and surface.
In the following, we demonstrate that a suitable choice of numerical fluxes and nonconservative terms recovers a split-form DGSEM that satisfies both well-balancedness and entropy stability. We apply a two-dimensional version of the EC fluxes \eqref{eq:ec_flux} in the volume, while we use the novel hydrostatic reconstruction with ES fluxes \eqref{eq:definition_es_flux} on the surface and use the same numerical non-conservative term \eqref{eq:nonconservative_flux} for both volume and surface fluxes
\begin{equation}
	\begin{aligned}
		\iprodN{\mathbb{I}^N(J)\statevec{U}_t,\statevecGreek{\varphi}} + 
		&\int\limits_{\partial E,N} \statevecGreek{\varphi}^T \{ \statevec{F}_{n,\epsilon}^{ES} - \statevec{F}_{n}\} \hat{s} \dS + \iprodN{\vec{\mathbb{D}}\cdot\bigcontravec{F}^{EC} ,\statevecGreek{\varphi}} 
		\\+
		&\int\limits_{\partial E,N} \statevecGreek{\varphi}^T \left(\statevecGreek{\Phi}\circ\bigstatevec{R}\right)^{\diamond}_{n,\epsilon} \hat{s} \,\text{d}S
		+
		\iprodN{\spacevec{\mathbb{D}}\cdot(\statevecGreek{\Phi}\circ\bigcontravec{R})^{\diamond},\statevecGreek{\varphi}} = 0,
	\end{aligned}
	\label{eq:DG_scheme}
\end{equation}
where the two-dimensional fluxes and nonconservative terms are
\begin{equation}
	\begin{aligned}
	&\statevec{F}_{m}^{EC,1} = \begin{pmatrix}
		\avg{h_{m}{v_m}} \\
		\avg{h_{m}v_m}\avg{v_m}\\
		\avg{h_{m}v_m}\avg{w_m}
	\end{pmatrix}, \quad 
	\statevec{F}_{m}^{EC,2} = \begin{pmatrix}
		\avg{h_{m}{w_m}} \\
		\avg{h_{m}w_m}\avg{v_m}\\
		\avg{h_{m}w_m}\avg{w_m}
	\end{pmatrix},\\
	&\statevec{F}_{m,\epsilon}^{ES} = \statevec{F}_{m,\epsilon}^{EC} - \frac{\lambda_{\max}}{2}\jump{\statevec{U_{\epsilon}}},
	\quad
	\left(\statevecGreek{\Phi}\circ\bigstatevec{R}\,\right)^{\diamond} = \statevecGreek{\Phi}_{ij} \circ \jump{\bigstatevec{R}}.
	\end{aligned}
	\label{eq:definition_2d_fluxes}
\end{equation}

\subsection{Entropy stability}
To show entropy stability for the split-form DGSEM, as in the FV case in Section \ref{sec:entropy_stability} we demonstrate that the scheme \eqref{eq:DG_scheme} satisfies the entropy inequality \eqref{eq:entropy_inequality} in a semi-discrete setting. 
Following the strategy from \cite{winters2021construction}, we first choose the test function $\statevecGreek{\varphi}=\statevec{W}$ to be the interpolated entropy variables, to contract the split-form DGSEM
\begin{equation}
	\begin{aligned}
		\iprodN{\mathbb{I}^N(J)\statevec{U}_t,\statevec{W}} + 
		&\int\limits_{\partial E,N} \statevec{W}^T \{ \statevec{F}_{n,\epsilon}^{ES} - \statevec{F}_{n}\} \hat{s} \dS + \iprodN{\vec{\mathbb{D}}\cdot\bigcontravec{F}^{EC} ,\statevec{W}} 
		\\+
		&\int\limits_{\partial E,N} \statevec{W}^T \left(\statevecGreek{\Phi}\circ\bigstatevec{R}\right)^{\diamond}_{n,\epsilon} \hat{s} \,\text{d}S
		+
		\iprodN{\spacevec{\mathbb{D}}\cdot(\statevecGreek{\Phi}\circ\bigcontravec{R})^{\diamond},\statevec{W}} = 0.
	\end{aligned}
	\label{eq:DG_scheme_entropy_space}
\end{equation}

As we are interested in the semi-discrete analysis, we assume time continuity and apply the chain rule such that the time derivative term in \eqref{eq:DG_scheme_entropy_space} contracts to the temporal rate of entropy change
\begin{equation}
	\iprodN{\mathbb{I}^N(J)\statevec{U}_t,\statevec{W}}
	=
	\sum\limits_{i,j=0}^{N} J_{ij}\omega_{ij}\statevec{W}_{ij}^T\frac{\operatorname{d}\statevec{U}_{ij}}{\operatorname{dt}}=
	\sum\limits_{i,j=0}^{N} J_{ij}\omega_{ij}\frac{\operatorname{d}S_{ij}}{\operatorname{dt}}
	= 
	\iprodN{\mathbb{I}^N(J)S_t,1}.
	\label{eq:discrete_entropy_analysis_time_derivative}
\end{equation}

We then sum \eqref{eq:discrete_entropy_analysis_time_derivative} over all elements $k=1,..,K$ in the mesh to obtain the total entropy change 
\begin{equation}
	\bar{S}_t = \sum_{k=1}^K \iprodN{\mathbb{I}^N(J)^kS_t^k,1}.
	\label{eq:total_discrete_entropy}
\end{equation}

Hence, summing \eqref{eq:DG_scheme_entropy_space} over all elements yields the total discrete entropy change of the split-form DGSEM and entropy stability follows for $\bar{S}_t \leq 0$. 
\begin{lemma}[Entropy-stability of the split-form DGSEM]
	The curvilinear split-form DGSEM
	\begin{equation}
		\begin{aligned}
			\iprodN{\mathbb{I}^N(J)\statevec{U}_t,\statevecGreek{\varphi}} + 
			&\int\limits_{\partial E,N} \statevecGreek{\varphi}^T \{ \statevec{F}_{n,\epsilon}^{ES} - \statevec{F}_{n}\} \hat{s} \dS + 	\iprodN{\vec{\mathbb{D}}\cdot\bigcontravec{F}^{EC} ,\statevecGreek{\varphi}} 
			\\+
			&\int\limits_{\partial E,N} \statevecGreek{\varphi}^T \left(\statevecGreek{\Phi}\circ\bigstatevec{R}\right)^{\diamond}_{n,\epsilon} \hat{s} 	\,\text{d}S
			+
			\iprodN{\spacevec{\mathbb{D}}\cdot(\statevecGreek{\Phi}\circ\bigcontravec{R})^{\diamond},\statevecGreek{\varphi}} = 0.
		\end{aligned}
		\label{eq:DG_scheme_proof_es}
	\end{equation}
	with the fluxes and nonconservative terms defined in \eqref{eq:definition_2d_fluxes} is entropy stable for the ML-SWE \eqref{eq:system_multilayer}.
\end{lemma}
\begin{proof}
	To show entropy stability we contract \eqref{eq:DG_scheme_proof_es} with entropy variables $\statevec{W}$ and first examine the volume integrals, assuming that the metric identities are satisfied \eqref{eq:metric_identity}. 
	Then we invoke the SBP property \eqref{eq:SBP_integration_by_parts} and entropy conservation condition \eqref{eq:entropy_conservation_condition} to apply Lemma 1 from \cite{ersing2024entropy}, which concludes that the volume contributions of the curvilinear split-form DGSEM with EC fluxes become the entropy flux evaluated at interfaces, when contracted into entropy space
	\begin{equation}
		\begin{aligned}
			\iprodN{\vec{\mathbb{D}}\cdot\bigcontravec{F}^{EC} ,\statevec{W}} + \iprodN{\spacevec{\mathbb{D}}\cdot(\statevecGreek{\Phi}\circ\bigcontravec{R})^{\diamond},\statevec{W}} = 		\int\limits_{\partial E,N} \left(\spacevec{F}^{\,S}\cdot\spacevec{n}\right) \hat{s} \dS.
		\end{aligned}
	\end{equation}
	Accordingly, the entropy balance in each element is governed solely by surface contributions.
	To evaluate the surface contributions we assume periodic boundaries and sum over all elements in the mesh $k=1,...K$, to obtain an expression for the total discrete entropy change \eqref{eq:total_discrete_entropy}
	\begin{equation}
			\bar{S}_t = -\sum_{k=1}^K\int\limits_{\partial E_k,N} \left(\statevec{W}^T \{ \statevec{F}_{n,\epsilon}^{ES} - \statevec{F}_n\} + \left(\spacevec{F}^{\,S}\cdot\spacevec{n}\right)
			+ \statevec{W}^T
			(\statevecGreek{\Phi}\circ\statevec{R})_{n,\epsilon}^{\diamond} 
			\right)
			\hat{s} \dS.  	
	\end{equation}
	Due to the nonconservative formulation of the pressure term, the physical flux contracts to the entropy flux $\statevec{w}^T\statevec{f} = {f}^S$ such that these terms cancel directly so that
	\begin{equation}
		\bar{S}_t =
		-\sum_{k=1}^K\int\limits_{\partial E_k,N} \left(\statevec{W}^T \statevec{F}_{n,\epsilon}^{ES}
		+ \statevec{W}^T
		(\statevecGreek{\Phi}\circ\statevec{R})_{n,\epsilon}^{\diamond} 
		\right)
		\hat{s} \dS.
	\end{equation}
	As our approximation is discontinuous, summing over the elements creates interface jumps and averages for all quantities besides the numerical surface fluxes, which are uniquely defined at the interface, to create the interface contributions
	\begin{equation}
			\bar{S}_t = 
			-\sum_{\text{faces}} \int\limits_{N}
			\left\{
			\jump{\statevec{W}}^T \bigstatevec{F}^{ES}_{\epsilon} 
			+
			\avg{\statevec{W}\circ\statevecGreek{\Phi}_{\epsilon}}^T\jump{\bigstatevec{R}_{\epsilon}}\right\}\cdot\spacevec{n}\hat{s}\dS.	
	\end{equation}
	Entropy stability then follows with an argument analogous to the one-dimensional case in Lemma 8, which shows that entropy is dissipated at each interface and accordingly $\bar{S}_t\leq 0$.
	\qed
\end{proof}

\begin{remark}
The two-dimensional proof of Lemma 8 follows the same reasoning as that in one-dimension, and is omitted here.
That is, the jump in the velocity in each Cartesian direction can be written on quadratic form similar to \eqref{eq:velocity_quadratic}.
The remainder of the proof, i.e. demonstrating that the inequality \eqref{eq:contraction_in_lemma8} holds, follows identical steps as in the one-dimensional case.
\end{remark}

\subsection{Shock-capturing}\label{sec:shock_capturing}
Even though the entropy stable method has improved robustness in comparison to standard DGSEM, additional shock-capturing is necessary to contain strong oscillations near discontinuities. This is especially important near wet/dry transitions as oscillations can lead to unphysical solutions with negative layer heights, that eventually terminate the computation. As a remedy, we apply the entropy stable subcell shock capturing method from Hennemann et al. \cite{hennemann2021provably}. 
This strategy introduces an element-wise blending at each node between the split-form DGSEM update $\dot{\statevec{U}}^{DG}_{ij}$ from \eqref{eq:DG_scheme} and a compatible FV update $\dot{\statevec{U}}^{FV}_{ij}$ from a mapped version of \eqref{eq:hydrostatic_reconstruction_scheme} on a sub-cell FV grid to obtain the hybrid scheme
\begin{equation}
	\dot{\statevec{U}}_{ij} = (1 - \alpha) \dot{\statevec{U}}^{DG}_{ij} + \alpha \dot{\statevec{U}}^{FV}_{ij},
\end{equation}
with a blending function $\alpha$ that determines a blending between the high- and low-order method, where $\alpha=0$ corresponds to a pure DG method and $\alpha=1$ to pure FV.
The blending function $\alpha$ is determined by a shock indicator that evaluates the modal energy of a suitable quantity chosen as $\sum_{m=1}^M \frac{1}{2}gh_m^3$, to determine under resolved elements. After an initial computation the blending function $\alpha$ is first clipped to maximum value $\alpha_{max}=0.5$ and then a smoothing 
\begin{equation}
\alpha^{final} = \underset{E}{\max}\{\alpha, 0.5\alpha_E\}
\end{equation}
over neighboring elements $E$ is applied to improve accuracy.
Complete details on the construction of the compatible subcell FV scheme and proofs concerning entropy stability and conservation are provided in \cite{hennemann2021provably}.

\subsection{Positivity-preservation}\label{sec:positivity_limiter}
To apply our method in the presence of wet/dry transitions, additional measures are necessary to ensure non-negative water heights. 
Therefore, we employ a positivity-preserving limiter following the strategy in \cite{xing2013positivity}.
The strategy to ensure positivity can be split into two distinct steps.
First, one uses the conservation properties of the split-form DGSEM in combination with a positivity-preserving numerical flux to ensure non-negative cell averages for the water heights $({\bar{h}_m})_{E}$ under a specific time step restriction using a simple forward Euler method in time. 
The requirement of a positivity-preserving numerical flux is satisfied as the numerical flux in the mass equations corresponds to the standard LLF flux. 
For the present discretization on curvilinear quadrilateral elements and LLF flux positivity was shown in \cite{bonev2018discontinuous} under the assumption of exact integration for the time step restriction
\begin{equation}
	\Delta t \leq \frac{\omega_1}{2\lambda_{max}}\frac{J_{E}}{J_{\partial E}},
	\label{eq:time_step_positivity}
\end{equation}
where $\omega_1$ is the first LGL quadrature weight and $J_{\partial E}$ the surface Jacobian.

Even though this ensures positivity of the cell average $(\bar{h}_m)_{E}$, the polynomial solution $h_m(\spacevec{x})$ may still be negative. 
In these cases a linear scaling limiter around the cell average ${\bar{h}_m}_{E}$ is applied to the polynomial $h_m(\spacevec{x})$ to ensure positivity. Therefore, at each node the water height $(h_m)_{ij}$ is replaced with a limited version
\begin{equation}
	(\tilde{h}_m)_{ij} = \theta_E \left( (h_m)_{ij} - (\bar{h}_m)_E \right) + (\bar{h}_m)_E,
\end{equation}
with the parameter $\theta_{m,E} \in [0,1]$ given by
\begin{equation}
	\theta_{m,E} = \min\left\{ 1, \frac{(\bar{h}_m)_E}{(\bar{h}_m)_E - \min((h_m)_{ij})} \right\},
\end{equation}
where for the ML-SWE, we apply the limiter separately for each layer $m$.

The positivity-limiter maintains both conservation and high-order accuracy \cite{xing2013positivity} and was shown to preserve entropy inequalities in \cite{wintermeyer2018entropy, ranocha2017shallow}. Analogous to Remark \ref{remark:SSPRK}, the limiter can be extended to SSPRK methods as these are a convex combination of forward Euler time steps.

\subsection{Well-balancedness}\label{sec:DG_well_balanced}
Another important aspect for the numerical treatment of shallow water flows is the preservation of steady state solutions, such as the lake-at-rest condition \eqref{eq:lake-at-rest_multilayer}.
Specifically for DG methods one needs to ensure both well-balancedness in the volume and surface integrals to preserve such steady state solutions.
A common issue in these methods is that the volume integral typically relies on the fact that the hydrostatic pressure remains constant.
Under wet/dry transitions this is not satisfied by the polynomial approximation and, without further treatment, well-balancedness is only achieved if the transition point is aligned with element interfaces, see, e.g., \cite{chan2022entropy, wintermeyer2018entropy, fu2022high}.
A remedy to this issue was presented in \cite{bonev2018discontinuous}, where the authors resort to a first order volume discretization in partially dry elements.

We adopt a similar approach and combine the positivity-preserving limiter, introduced in Section \ref{sec:positivity_limiter}, with the low-order hydrostatic reconstruction scheme \eqref{eq:hydrostatic_reconstruction_scheme} in the presence of partially dry elements to ensure positivity and well-balancedness for arbitrary transition locations.
To achieve this we modify the shock-capturing indicator and set it to $\alpha=1$, in order to switch to a pure FV method in partially dry elements. 
We employ a simple threshold and determine elements to be partially dry if the water height is below a given threshold $\tau_{wet}$.

Even though well-balancedness only requires the FV method at dry nodes, switching to FV for small water heights has shown improved robustness for the scheme.
While this incurs some accuracy degradation, the total influence is considered low for reasonably small water heights as the switch to low-order remains localized to the wet/dry front.

In the following we demonstrate that this approach yields a well-balanced formulation in both wet and dry domains. 
Therefore, we first prove that the volume integral in \eqref{eq:DG_scheme} is well-balanced on fully wet elements, before we conclude to prove well-balancedness in case of dry elements. 
\begin{lemma}[Well-balancedness of the curvilinear split-form DGSEM volume integral] \label{lemma:wb_volume_integral}
	The volume integral of the curvilinear split-form DGSEM \eqref{eq:DG_scheme}
	\begin{equation}
		\begin{aligned}
			\iprodN{\vec{\mathbb{D}}\cdot\bigcontravec{F}^{EC} ,\statevecGreek{\varphi}} 
			+
			\iprodN{\spacevec{\mathbb{D}}\cdot(\statevecGreek{\Phi}\circ\bigcontravec{R})^{\diamond},\statevecGreek{\varphi}}
		\end{aligned}
	\end{equation}
	satisfies the lake-at-rest condition \eqref{eq:lake-at-rest_multilayer} on fully wet domains.
\end{lemma}
\begin{proof}
	Well-balancedness in the mass equation, follows directly from the initial condition $v_m, w_m = 0$, as all products with the velocity vanish.
	Thus, we only need to examine the momentum equation.
	As the pressure contributions were moved to the nonconservative term, again the conservative volume integral vanishes from the initial condition $v_m, w_m = 0$.
	For the nonconservative volume integral we use that all terms with local contributions $(\statevecGreek{\Phi}_{ij}\circ\bigstatevec{R}_{ij})$ vanish in the volume integral if the metric identities \eqref{eq:metric_identity} are satisfied to obtain
	\begin{equation}
		\begin{aligned}
			\spacevec{\mathbb{D}}\cdot(\statevecGreek{\Phi}\circ\bigcontravec{R})^{\diamond} = 
			&\sum_{l=0}^{N} \mathcal{D}_{il} \left(\statevecGreek{\Phi}_{ij}\circ\left(\bigstatevec{R}_{lj} \cdot \avg{J\vec{a}^{\,1}}_{(i,l)j} \right)\right)
			+
			&\sum_{l=0}^{N} \mathcal{D}_{jl} \left(\statevecGreek{\Phi}_{ij}\circ\left(\bigstatevec{R}_{il} \cdot \avg{J\vec{a}^{\,2}}_{i(j,l)} \right)\right).
		\end{aligned}
	\end{equation}

	We then expand the nonconservative terms, for the $h_mv_m$-equation and
	use that $b + \sum\limits_{k\geq m} h_k + \sum_{k<m} \sigma_{km}h_k$ is constant in the wet domain. Then assuming that metric identities are satisfied we obtain
	\begin{equation}
		\begin{aligned}
			&\iprodN{\spacevec{\mathbb{D}}\cdot(\statevecGreek{\Phi}\circ\bigcontravec{R})^{\diamond},\statevecGreek{\varphi_{hv_m}}}
			\\&=
			g\sum\limits_{i,j=0}^N \omega_{ij} \left(h_m\right)_{ij}
			\left[\sum_{l=0}^{N} \mathcal{D}_{il} \left(\left(b + \sum\limits_{k\geq m} h_k + \sum_{k<m} \sigma_{km}h_k \right)_{lj} \avg{J\vec{a}_1^1}_{(i,l)j} \right) 
			\right.\\&\hspace{2.26cm}+\left.
			\sum_{l=0}^{N} \mathcal{D}_{jl}   \left(\left( b + \sum\limits_{k\geq m} h_k + \sum_{k<m} \sigma_{km}h_k \right)_{il} \avg{J\vec{a}_1^2}_{i(j,l)} \right)\right]
			\\&=
			g \left(b + \sum\limits_{k\geq m} h_k + \sum_{k<m} \sigma_{km}h_k \right) \sum\limits_{i,j=0}^N \omega_{ij} \left(h_m\right)_{ij}
			\left[\sum_{l=0}^{N} \mathcal{D}_{il} \left(\left(J\vec{a}_1^1\right)_{ij} +\left(J\vec{a}_1^1\right)_{lj} \right) 
			+
			\sum_{l=0}^{N} \mathcal{D}_{jl}   \left(\left(J\vec{a}_1^2\right)_{ij} +\left(J\vec{a}_1^2\right)_{il} \right)\right]
			\\&=
			g \left(b + \sum\limits_{k\geq m} h_k + \sum_{k<m} \sigma_{km}h_k \right) \sum\limits_{i,j=0}^N \omega_{ij} \left(h_m\right)_{ij}
			\left[\sum_{l=0}^{N} \mathcal{D}_{il} \left(J\vec{a}_1^1\right)_{lj}
			+
			\sum_{l=0}^{N} \mathcal{D}_{jl}   \left(J\vec{a}_1^2\right)_{il} \right] = 0,
		\end{aligned}
	\end{equation}
	which completes the proof. \qed
\end{proof}

With this result we are then prepared to extend the well-balancedness result to the general case of both wet or dry domains, by resorting to pure FV subcells in the presence of dry nodes.

\begin{lemma}[Well-balancedness of the curvilinear split-form DGSEM]
	The curvilinear split-form DGSEM \eqref{eq:DG_scheme} with the positivity-preserving limiter (Section \ref{sec:positivity_limiter}), shock-capturing (Section \ref{sec:shock_capturing}) and the well-balanced treatment (Section \ref{sec:DG_well_balanced}) preserves the lake-at-rest condition \eqref{eq:lake-at-rest_multilayer}.
\end{lemma}
\begin{proof}
	To prove well-balancedness in the more general case, we consider wet and dry (or partially dry) domains separately. 
	
	If an element is partially dry, the well-balanced treatment introduced in Section \ref{sec:DG_well_balanced} switches locally to FV subcells. Well-balancedness for the FV scheme has been shown for both wet and dry domains in Lemma \ref{lemma:wb_ml}.
	
	In the wet domain, the shock-capturing may introduce a blending between the FV and DG solutions. Well-balancedness of the resulting hybrid-scheme follows if both contributions vanish independently under lake-at-rest conditions.
	For the FV scheme this follows again from Lemma \ref{lemma:wb_ml}. 
	In case of the DG scheme, both volume and surface contributions need to be considered to obtain well-balancedness.
	In wet domains well-balancedness in the volume integral was shown in Lemma \ref{lemma:wb_volume_integral}. 
	On the surface, the lake-at-rest condition \eqref{eq:lake-at-rest_multilayer} ensures that the physical flux contributions vanish with the velocity.
	The remaining surface numerical flux and nonconservative term are then identical to the FV scheme and well-balancedness follows again from Lemma \ref{lemma:wb_ml}.
	\qed
\end{proof}

\subsection{Numerical treatment for vanishing water heights}\label{sec:numerical_treatment}
While the positivity-limiter ensures positive water heights an additional treatment is necessary to avoid numerical issues for vanishing water heights and the presence of dry cells due to finite precision.
The main concern regarding numerical issues at dry nodes is the computation of the velocities, which are recovered from a division of conserved variables
\begin{equation}
	v_m = \frac{hv_m}{h_m}.
\end{equation}
In the presence of vanishing water heights this approaches a division by zero, with severe effects on robustness and numerical accuracy. 
To prevent these issues, we ensure the presence of a finite water height together with a desingularization formula for the velocity. 

For practical computations we initialize the solution with a minimum water height to avoid completely dry nodes, where solutions to \eqref{eq:system_multilayer} are not defined. Therefore, at each nodal value we apply the function
\begin{equation}
	h_m = \max\{h_m, 5\epsilon_{num}\}
	\label{eq:fix_water_level}
\end{equation}
where $\epsilon_{num}$ denotes machine epsilon using double precision and consider these nodes as dry.
To retain the minimum water level \eqref{eq:fix_water_level} is applied after application of the positivity limiter, to prevent draining of dry nodes. 
We deliberately chose the threshold as small as possible, since the well-balanced strategy introduced in Section \ref{sec:well-balancedness} relies on a vanishing water heights at wet/dry transitions. 
Numerically, this cannot be satisfied exactly, which introduces minor perturbations around machine precision from the lake-at-rest condition in configurations with a wet/dry transition. 
However, numerical results demonstrate that this has only negligible influence on the well-balanced property of the approximation.

While the latter procedure effectively prevents a singularity in the velocity computation, we find that additional treatment is necessary to ensure robustness of the velocity computation. We therefore either apply the desingularization formula from \cite{chertock2015well} to adjust the momentum at each node or set the momentum to zero at dry nodes 
\begin{equation}
	hv_m = 
	\begin{cases}
		\frac{2h_m^2hv_m}{h_m^2 + \max(h_m^2, \tau_{vel})} & \text{if } h_m > 5\epsilon_{num}\\
		0	& \text{else},
	\end{cases}
\end{equation}
where the threshold $\tau_{vel}$ is some small positive value, dependent on the problem setup. This procedure recovers the exact momentum, for $h_m^2 \geq \tau_{vel}$ and ensures a robust velocity computation and vanishing velocity at dry nodes.

\subsection{Pseudocode algorithm}
In the following, we present a pseudocode algorithm to outline how the different steps of the wetting and drying procedure are incorporated into the split-form DGSEM.
The algorithm describes the wetting and drying procedure for the split-form DGSEM discussed in Section \ref{sec:DG_method}, within one time step of a SSPRK method.
To focus on the main ideas, we keep the algorithm simple and focus on the one-dimensional SWE. An extension to two-dimensions and multiple layers follows the same outline and the actual implementation is available in \cite{ersing2024hrRepro}.
For the present algorithm, we introduce a global index $i$ over all nodes in the domain.
We further introduce superscript $n$ to denote the $n$th time step and assume a given time step for the SSPRK method.

\begin{algorithm}[H]
		\KwIn{$h_i^n,\, hv_i^n,\, b_i$, $\tau_{vel}$, $\tau_{wet}$ \tcp{conservative variables and thresholds at time $t_n$}}
		\KwOut{$h_i^{n+1},\, hv_i^{n+1},\, b_i$ \tcp{conservative variables at time $t_{n+1}$}}\;
		\vspace{-0.35cm}
		\For{stage s in SSPRK method}{
			\For{element E in mesh}{
				Compute shock-capturing indicator $\alpha_E$ (Section \ref{sec:shock_capturing})\;
				\tcp{Modification to introduce pure FV cells in partially dry elements}
				\If{$\exists j \in E\,:\,h_j^n \leq \tau_{wet}$}{$\alpha_E = 1$}
			}	
		\For{interface k in mesh}{
		Apply hydrostatic reconstruction procedure (equation \eqref{eq:reconstruction_bottom} - \eqref{eq:reconstruction_conservative_variables}) to obtain $\left(\statevec{u}^-_\epsilon\right)^n_k$ and $\left(\statevec{u}^+_\epsilon\right)^n_k$\;
		
		Evaluate numerical fluxes $(\statevec{f}^{*}_{\epsilon})_k^n$ and nonconservative terms $\left(\left(\statevecGreek{\phi}\circ\statevec{r}\right)^{\diamond}_{\epsilon}\right)^n_k$
		}
		Execute split-form DGSEM with entropy-stable subcell blending (Section \ref{sec:DG_method})\;
		Perform the SSPRK stage $s$ update to obtain $h_i^{n+1},\,  hv_i^{n+1}$\;
		Apply the positivity-preserving limiter (Section \ref{sec:positivity_limiter})\;
		\tcp{Numerical treatment for vanishing water heights (Section \ref{sec:numerical_treatment})}
		\For{element $E$ in mesh}{
			\For{node $j$ in $E$}{
				\eIf{$h_j^{n+1}$ $\leq 5\epsilon_{\text{num}}$}{
					$h_j^{n+1} = 5\epsilon_{\text{num}}$\;
					$hv_j^n{+1} = 0$}
					{$hv_j^{n+1} = \frac{2(h^{n+1}_j)^2 hv^{n+1}_j}{(h_j^{n+1})^2 + \max((h_j^{n+1})^2, \tau_{vel})}$}
				}
			} 
		}
		\caption{Wetting and drying procedure for the split-form DGSEM}
		\label{alg:Algorithm1}
\end{algorithm}

\section{Results}\label{sec:results}
In the following, we present numerical results to demonstrate convergence, well-balancedness, pos\-i\-tiv\-i\-ty-pre\-serv\-a\-tion and entropy stability of the high-order reconstruction method with ES fluxes.
These numerical results were obtained using the open-source framework Trixi.jl \cite{ranocha2022adaptive,schlottkelakemper2021purely} and visualized with Makie.jl \cite{Makie2021}. A reproducibility repository for this section is available on Zenodo and GitHub \cite{ersing2024hrRepro}.
For the computations we use a four-stage, third-order SSPRK method with CFL-based time stepping.
We set the thresholds for partially-dry cells to $\tau_{wet}=10^{-4}$ and for desingularization to $\tau_{vel}=10^{-8}$, which gave good results in our testing.
All tests were performed on a single core of a 64-bit notebook with \textit{AMD Ryzen™ 7 PRO 4750U} processor and $16\,$GB main memory.
For each test case, we report required CPU times on our test system, using the same setup available in the reproducibility repository \cite{ersing2024hrRepro}.

	\subsection{Convergence}
		First, we demonstrate spectral convergence of our method by evaluating the error of our numerical approximation in the 	$L_2$-norm relative to an exact solution using the method of manufactured solutions.
		
		The manufactured solution for this three-layer system is constructed around trigonometric functions for layer heights and bottom topography and constant velocities
		\begin{equation}
			\begin{aligned}
				H_1 &= 4 + 0.1\cos\left(2\pi x+t \right) + 0.1\cos\left(2\pi y + t\right),
				\\
				H_2 &= 2 + 0.1\sin\left(2\pi x+t \right) + 0.1\sin\left(2\pi y + t\right),
				\\
				H_3 &= 1.5 + 0.1\cos\left(2\pi x+t \right) + 0.1\cos\left(2\pi y + t\right),
				\\
				b &= 1 + 0.1\cos\left(2\pi x \right) + 0.1\cos\left(2\pi y\right),
				\\
				v_m &= 0.8, \quad w_m = 1.0, \quad m=1,2,3,
			\end{aligned}
		\end{equation}  
		with gravity constant set to $g=1.1$ and layer densities $\rho_1 = 0.9$, $\rho_2 = 1.0$, $\rho_3= 1.1$.
		To obtain the manufactured solution, the respective source terms were then calculated with symbolic differentiation \cite{Symbolics.jl}.
		
		The problem is then computed on the $4 \times 4$ element curvilinear mesh shown in Figure~\ref{fig:WB_Curvilinear} with periodic boundary conditions, where edges are interpolated with polynomial degree $N=6$.
		To investigate spectral convergence we compute solutions up to polynomial degree $N=24$ at final time $t_{end}=0.1$ with fixed time step $\Delta t = 10^{-4}$.	
		In Figure~\ref{fig:Convergence_results} we present convergence results, by evaluating the $L_2$-Errors for $h_1$, $h_2$ and $h_3$. 
		For all quantities we observe spectral convergence up to degree $N=20$, with a characteristic odd-even pattern.
		
		\begin{figure}
			\centering
			\includegraphics[width=0.5\textwidth]{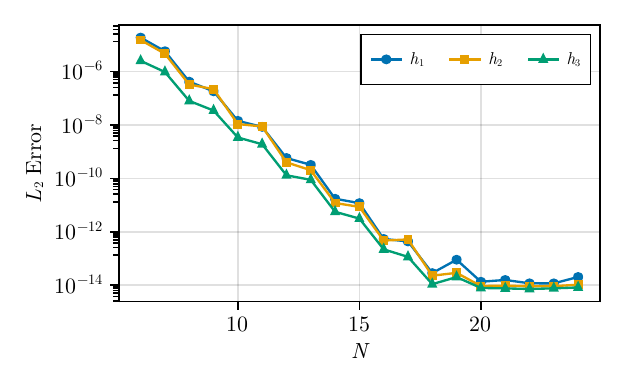}
			\caption{Spectral convergence results in space for the $L_2$-Error in quantities $h_1$, $h_2$ and $h_3$ over polynomial degree $N$ and $\Delta t=10^{-4}$}
			\label{fig:Convergence_results}
		\end{figure}
		
	\subsection{Well-balancedness}
		Next, we demonstrate the well-balancedness and entropy stability of our method for a two-layer system with complex bottom topography and wetting and drying.
		
		This test was introduced in \cite{martinez20201d} and applies a complex bottom topography to test well-balancedness under various wet/dry configurations.
		The initial condition together with a temporal evolution is shown in Figure~\ref{fig:wb_setup_2layer}.
		
		\begin{figure}[htb]
			\centering
			\includegraphics[width=0.8\textwidth]{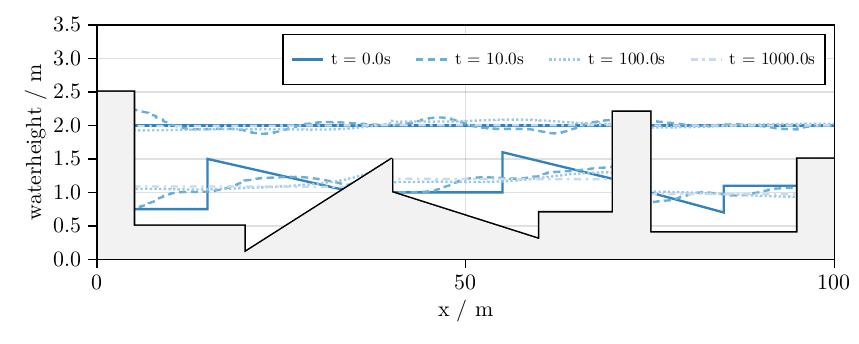}
			\caption{Initial condition and temporal evolution of the free surface for a two-layer system with perturbed lake-at-rest setup with $N=1$, $\text{CFL}=0.7$ on $100$ equidistant elements.}
			\label{fig:wb_setup_2layer}
		\end{figure}
		
		The test is setup with a density ratio $\sigma_{12} = \nicefrac{1}{3}$, gravity constant $g=1$, wall boundary conditions on both sides and a perturbation in the lower layer. 
		While the initial solution does not satisfy the lake-at-rest, the perturbation will be dissipated over time by numerical dissipation such that a well-balanced scheme should converge to steady state.
		Besides the steady state, we also use this test to demonstrate entropy stability as the complex bottom topography covers various wet/dry configurations.
		
		For this test we split the domain into $100$ elements and set a low polynomial degree $N=1$ in order to introduce enough dissipation to reach the steady state within a reasonable amount of time $t_{end}=12\,000$, which took $5.99s$ of CPU time on the test system.
		
		In Figure~\ref{fig:time_series_wb_2l} we present time series for the lake-at-rest error and total entropy change.
		In addition to the perturbed initialization, we also show computed quantities for the respective steady state initial condition.
		
		For the steady state initialization, we observe that the lake-at-rest is preserved around machine precision.
		In the perturbed case, the initial lake-at-rest error shows an exponential decay as the perturbation gets dissipated over time.
		Finally, the lake-at-rest error converges at around $10^{-12}$ at time $t=10\,000$, which demonstrates that the method remains well-balanced even after considerable time.
		On the other hand, it also shows that well-balancedness is not satisfied exactly, but we see a small deviation from the expected lake-at-rest. 
		A possible reason for this is that the well-balanced proofs assume zero water heights at dry interfaces, which we only satisfy asymptotically in the discrete case.
		However, since deviations remain reasonably small even for considerable time, we consider the effect to be negligible.
		
		Besides the lake-at-rest we also evaluate the total entropy change within each time step. 
		As shown in Figure~\ref{fig:time_series_wb_2l} for both steady state and perturbed initializations entropy is dissipated in each time step, which demonstrates the entropy stability of our method.
		\begin{figure}[htb]
		\centering
		\includegraphics[width=\textwidth]{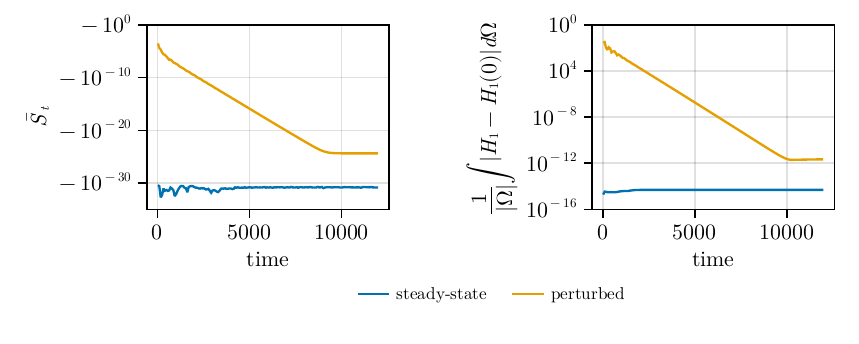}
		\caption{Semi-log plots for total entropy change and lake-at-rest error in $H_1$ over time for the perturbed and steady state lake-at-rest initial condition with $N=1$, $\text{CFL}=0.7$ on $100$ equidistant elements.}
		\label{fig:time_series_wb_2l}
		\end{figure}
		
	\subsection{Well-balancedness (curvilinear)}
		In order to show that well-balancedness is still satisfied in the curvilinear case, we present a second well-balanced test in two-dimensions.
		For this test we initialize the following lake-at-rest setup for a three-layer system with discontinuous bottom 	topography and dry states in each layer
		\begin{equation}
			\begin{aligned}
			H_1 = 1.5, \quad H_2 = 1.0, \quad H_3 = 0.5, \quad
			v_m = 0, \quad w_m = 0,  \qquad m=1,2,3
			\\
			b_0 = 0.2 + 0.1\sin\left({2\pi x}\right) + 0.1\cos\left({2\pi y}\right), \quad 
			b = 
			\begin{cases}
				b_0 + 0.1, & \text{for element } 3 \times 3\\
				b_0 + 0.5, & \text{for element } 2 \times 3\\
				b_0 + 1.0, & \text{for element } 2 \times 2\\
				b_0 + 1.5, & \text{for element } 3 \times 2\\
				b_0, & \text{else}.
			\end{cases}
			\end{aligned}
		\end{equation}
	
		The curvilinear mesh with the bottom topographyis depicted in Figure~\ref{fig:WB_Curvilinear}.
		In Table~\ref{tab:curvilinear_lake_at_rest_errors} we present computed lake-at-rest errors for all layers at final time $t_{end}=200$ using polynomial degree $N=6$. 
		Execution of this test case required $54.6s$ of CPU time on our test system.
		The results show that the lake-at-rest errors remain around machine precision for all layers, which demonstrates that the well-balanced property holds even in the curvilinear case.
		
		\begin{figure}[htb]
			\centering
			\includegraphics[width=0.5\textwidth]{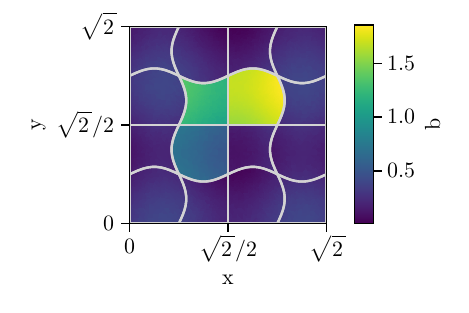}
			\caption{Contour plot of the bottom topography on the curvilinear mesh for the three-layer lake-at-rest setup.}
			\label{fig:WB_Curvilinear}
		\end{figure}
		\begin{table}
			\caption{Lake-at-rest errors for the three-layer curvilinear lake-at-rest setup at final time $t_{end}=200$, obtained with $N=6$ and $\text{CFL}=1$.}
			\label{tab:curvilinear_lake_at_rest_errors}
			\vspace{1mm}
			\centering
			\begin{tabular}{c|c|c}
			$\frac{1}{|\Omega|}\int H_1 - H_1(0)\,d\Omega$ & $\frac{1}{|\Omega|}\int H_2 - H_2(0)\,d\Omega$ & $\frac{1}{|\Omega|}\int H_3 - H_3(0)\,d\Omega$ \\[0.5ex] \hline\hline &&\\[-2ex]
			$2.032\cdot10^{-14}$ & $1.115\cdot10^{-14}$  & $4.015\cdot10^{-15}$
			\end{tabular}
		\end{table}
	
	\subsection{Dam break over triangular bottom}
		In the next test case, we consider a 1D dam break flow over triangular bottom topography. This setup was investigated experimentally in the CADAM project \cite{morris2000concerted}.
		
		The test was conducted in a $38\,m$ long and $1.75\,m$ wide channel, with a gate located at $15.5\,m$ and solid walls at each end.
		Initially, there is a reservoir with water height $0.75\,m$ upstream of the gate and another reservoir with height $0.15\,m$ downstream of the $6\,m$ long and $0.4\,m$ high triangular obstacle.
		To validate our numerical results we compare time series data of the water height with experimental data gathered at four gauge points, located $4\,m$ (G4), $10\,m$ (G10), $13\,m$ (G13) and $20\,m$ (G20) downstream of the dam gate. 
		The initial setup and a temporal evolution of the free surface are shown in Figure~\ref{fig:triangular_dam_break}. 
		\begin{figure}[htb]
			\centering
			\includegraphics[width=0.8\textwidth]{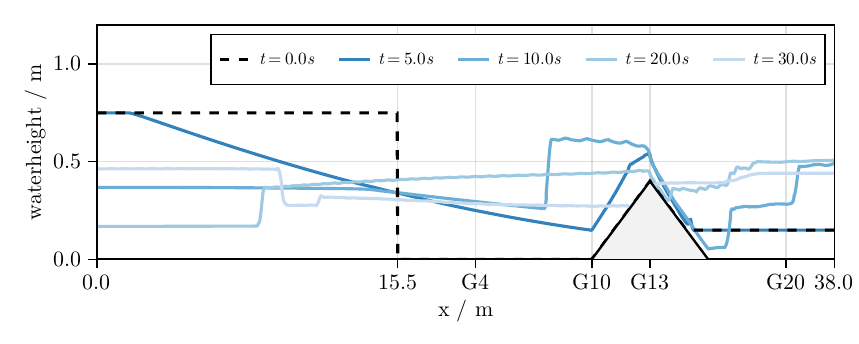}
			\caption{Initial condition and temporal evolution of the free surface for the triangular dam break setup at 	times $t=\{0s,5s,10s,20s,30s\}$ obtained with $N=4$, CFL=0.7 on 128 equidistant elements.}
			\label{fig:triangular_dam_break}
		\end{figure}
		
		For this test case we set a polynomial degree $N=4$ and discretize the domain with $128$ equidistant elements. 
		The solution is then computed up to time $t_{end} = 40\,s$ with a time step determined by $\text{CFL}=0.7$, which took $575\,ms$ of CPU time on our test system.
		To account for bottom friction and ensure correct comparison with experimental data, we further introduce an additional source term for the Manning friction with Manning's coefficient set to $n=0.0125\,{s}\,{m^{-1/3}}$.
		
		In Figure~\ref{fig:time_series_gauges} we compare the time series from both experimental data with our numerical results for the water height at gauge points G4, G10, G13 and G20.
		
		After the initial instantaneous dam break an initial flood wave propagates downstream and creates an initial rise in the water level at gauge points G4 and G10.  
		The flood wave then reaches the obstacle and overtopping starts to occur at G13 around $4\,s$.
		The wave run up causes a partial reflection, which causes a wave propagating upstream that leads to a another sudden rise in the water level at G4 and G10.
		On the downstream side, the overtopping creates an inflow to the downstream reservoir.
		Here the flow is reflected at the downstream wall and leads to a second wave overtopping the crest (G13) in downstream direction around $22.5\,s$ and subsequently propagates to gauge points G10 and G4. 
		
		At all measurement locations we observe very good agreement of our numerical results with experimental data. 
		In particular, we see an accurate prediction of the arrival times through all gauge locations.
		The water height on the other hand is slightly over predicted upstream of the obstacle and underestimated in the 	downstream reservoir.
		Very similar results were obtained in \cite{gu2017swe}, where the authors mention the missing vertical 	acceleration of the SWE as a reason for the discrepancies.
		
		\begin{figure}[htb]
			\centering
			\includegraphics[width=0.75\textwidth]{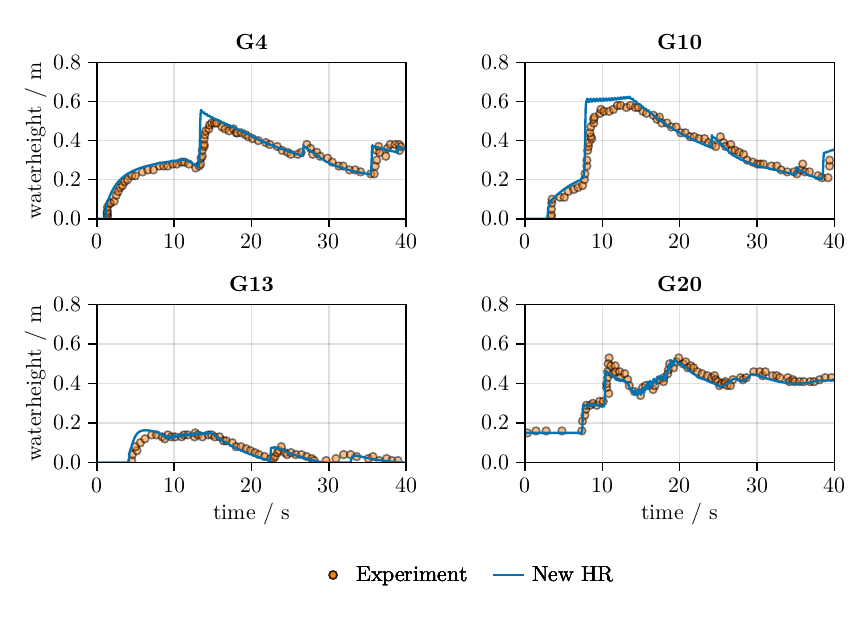}
			\caption{Comparison of computed water height against experimental data at different Gauge locations 	obtained with $N=4$, $\text{CFL}=0.7$ on $128$ equidistant elements.}
			\label{fig:time_series_gauges}
		\end{figure}
	
	\subsection{2D oscillating lake}
	To further evaluate the method under dynamic wetting and drying conditions we consider the two-dimensional oscillating lake.
	An analytical solution for this test was presented in \cite{Thacker1981} and we investigate the specific setup that was proposed in \cite{Gallardo2007} and also considered in \cite{wintermeyer2018entropy,Vater2015}.
	The test problem consists of a parabolic bowl in the domain $\Omega = [-2,2]^2$ with bottom topography 
	\begin{equation}
		b(x,y) = h_0 \frac{x^2 + y^2}{a^2},
	\end{equation}
	that contains an oscillating lake given by
	\begin{equation}
		\begin{aligned}
			h(x,y,0) &= \max \left\{0, \; \sigma \frac{h_0}{a^2} \left(2x \cos{\omega t} + 2y \sin{\omega t} - \sigma \right) + h_0 - b\right\},\\
			v(x,y,0) &= -\sigma \omega \sin \omega t,\\
			w(x,y,0) &= \sigma \omega \cos \omega t, 
		\end{aligned}
	\end{equation}
	with parameters $h_0=0.1$, $a=1$, $\sigma = 0.5$, $g = 9.81$ and $\omega = {\sqrt{2gh_0}}/{a}$.
	For the computations we use a Cartesian mesh with $100 \times 100$ elements and polynomial degree $N=3$.
	The solution is computed until a final time $t=2T$ using $\text{CFL}=0.7$, where $T = \frac{2\pi}{\omega}$ denotes the oscillation period of the analytical solution.
	We set periodic boundary conditions as the boundary elements remains dry throughout the computation. The computation took $386\,$s of CPU time on our test system.
	
	To evaluate the method on this test case, we present results for the water level and momentum at times $t=\{T+T/3, 2T\}$ in Figure \ref{fig:Oscillating_lake}.
		\begin{figure}[htb]
		\centering
		\subfigure{
			\includegraphics[width=0.4\textwidth]{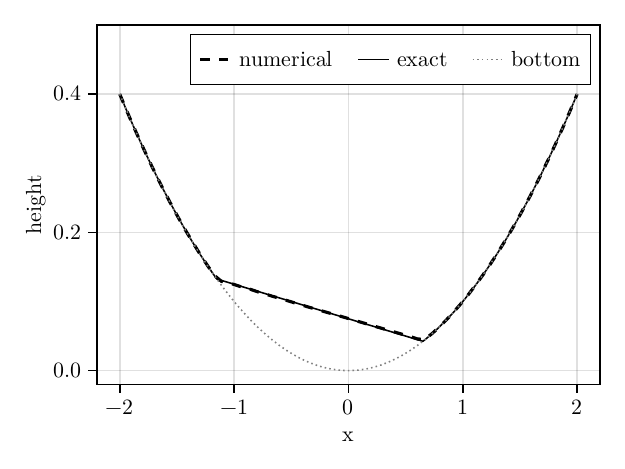}
		}
		\subfigure{
			\includegraphics[width=0.4\textwidth]{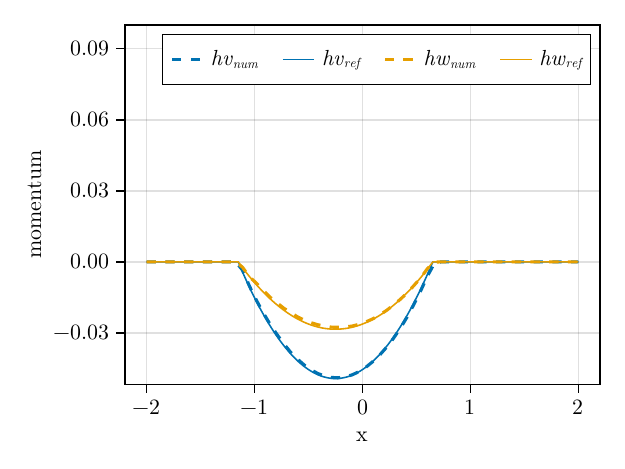}
		}
		\subfigure{
			\includegraphics[width=0.4\textwidth]{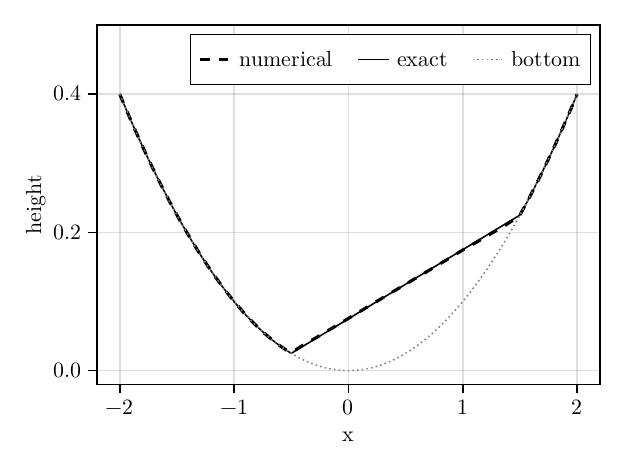}
		}
		\subfigure{
			\includegraphics[width=0.4\textwidth]{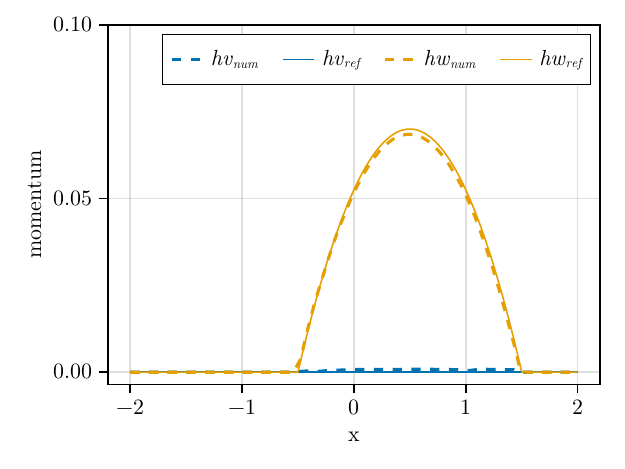}
		}
		\caption{Comparison of water level and momentum against an analytical solution for the oscillating lake at times $t=T+T/3$ (upper row) and $t=2T$ (lower row) along the slice $y=0$. Results were obtained with polynomial degree $N=3$ and $CFL=0.7$ on a Cartesian mesh with $100\times100$ elements.}
		\label{fig:Oscillating_lake}
	\end{figure} 
	Here, we present the numerical solution together with the analytical solution at a slice through $y=0$.
	The results demonstrate an accurate prediction for both the water level and momentum. 
	Despite the dynamically moving wet/dry fronts in this test case, the solution remains free of oscillations, which demonstrates the robustness of the high-order method under wetting and drying.
	We observe a slightly increasing deviation in both water level and momentum compared to the analytical solution at the later time $T=2T$, which could be reduced under further spatial refinement.
	
	\subsection{Lock exchange problem}
	Next, we consider a two-layer lock exchange problem, that was also investigated in \cite{kurganov2009central, BENKHALDOUN2014, fernandez2011intermediate}.
	In this test case an exchange flow of two fluids with density ratio $\rho_1 / \rho_2 = 0.98$ over a Gaussian shaped bottom topography
	\begin{equation}
    	b(x) = e^{-x^2} - 2,
	\end{equation}
	is investigated within the domain $\Omega = [-3, 3]$.
	Initially, the two fluids are assumed to be at rest ($v_1(x,t=0) = v_2(x,t=0) = 0$) and separated at the middle of the domain.
	The layer heights at $t=0$ are given by
	\begin{equation}
		\begin{cases} h_1 = -b,\; h_2 = 0, & \text{for} \quad -3 \leq x \leq 0, \\ h_1 = 0,\; h_2 = -b, & \text{for} \quad 0 < x \leq 3 , \end{cases}
	\end{equation}
	such that each layer is initialized with a wet state in one side of the domain and a dry side on the other, separated by a discontinuity.
	At both boundaries the condition $hv_1 = - hv_2$ is imposed to represent exchange with an external reservoir.
	A density-driven flow develops from these initial conditions due to the gravitational acceleration $g=9.81$, where the lighter fluid on top flows to the right and the heavier fluid below flows to the left of the domain until a steady state is reached.
	
	For the computations we discretize the domain $\Omega$ into $128$ uniform elements and set the polynomial degree to $N=1$ or $N=2$. 
	We compute the solution up to a final time $t_{end} = 100$ with $\text{CFL = 0.7}$, which took $11.3\,$s for $N=1$ and $20.7\,s$ for $N=2$ of CPU time on our test system.
	The interface heights computed with both $N=1$ and $N=2$ at the final time are shown in Figure $\ref{fig:lock_exchange}$ together with an indicator for the loss of hyperbolicity
	\begin{equation}
		I_{hyp} = \frac{(v_1 - v_2)^2}{g(1- \rho_1 / \rho_2)\left(h_1 + h_2\right)},
	\end{equation}
	 where the two-layer system is considered hyperbolic if $I_{hyp} \leq 1$ and non-hyperbolic if $I_{hyp}>1$, as the approximated eigenvalues become complex in the latter case.
	 For a detailed discussion on the issue of hyperbolicity loss, see \cite{diaz2023discussion}.
	 
	 At the final time both solutions tend to a steady state and show good agreement with the results in \cite{kurganov2009central, BENKHALDOUN2014, fernandez2011intermediate}.
	 However, for solutions with polynomial degree $N=2$ and higher, we observe the onset of spurious oscillations, which are restricted to the non-hyperbolic region. 
	 The oscillations first develop and quickly propagate from right to left for flow near the middle of the domain around $t=2.8$.
	 They appear to originate from the transition point between hyperbolicity and non-hyperbolicity.
	 Similar spurious oscillations connected to the loss of hyperbolicity have been observed for other experiments in \cite{diaz2023discussion, martinez20201d}.
	 Remarkably, even with the presence of these oscillations, the moving wet/dry fronts do not cause any positivity issues, which demonstrates the robustness of the method.
	 Interestingly, the inherent dissipation of lower order methods is sufficient in this test case to dampen the spurious oscillations, while the higher order method is less diffusive and reveals the underlying issue of the hyperbolicity loss.
	 The results demonstrate that in order to obtain an accurate high order approximation for problems in the non-hyperbolic domain, additional treatments for the hyperbolicity loss need to be considered. 
	\begin{figure}[htb]
		\centering
		\includegraphics[width=0.9\textwidth]{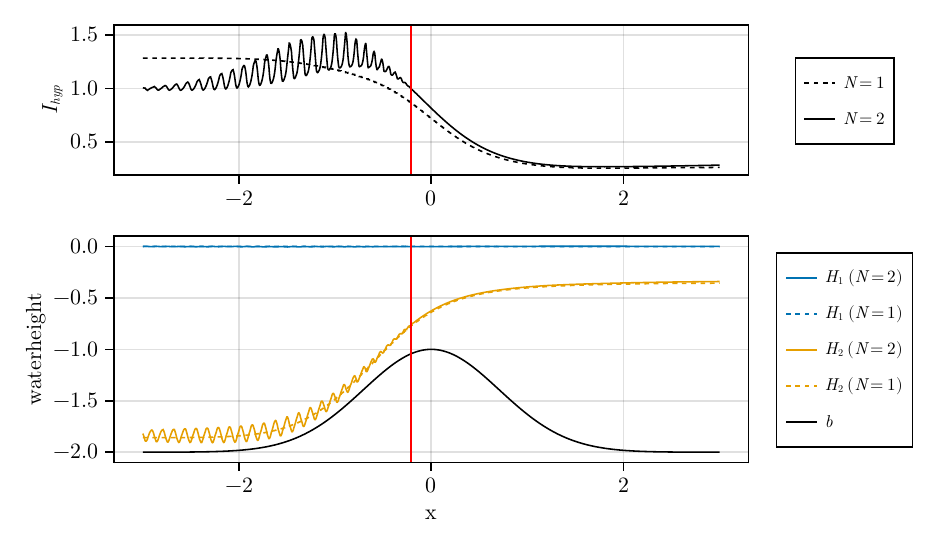}
		\caption{Hyperbolicity indicator $I_{hyp}$ and interface heights for the lock exchange problem at final time $t_{end}=10\,s$ for different polynomial degrees $N=\{1,2\}$. The red line marks the transition towards the non-hyperbolic regime for the case $N=2$.}
		\label{fig:lock_exchange}
	\end{figure}
	
	\subsection{Multilayer dam break}
		Finally, we present results for a more complex dam break setup with three layers and a bottom topography on a curvilinear mesh. 
		The problem is computed on the domain $\Omega=[-1,1]^2$ on an unstructured curvilinear mesh with $20\times 20$ elements.
		The following initial condition for a dam break over initially dry bottom topography is prescribed
		\begin{equation}
			\begin{aligned}
			&\begin{cases}
				H_1 = 1.0, \quad H_2 = 0.8,  \quad H_3 = 0.6 & \text{if } x<0.0, \\
				H_1, H_2, H_3 = 0.0 & \text{else},
			\end{cases}
			\\
			&b = 1.4e^{-10(x^2+y^2)},\quad 
			v_m = 0.8, \quad w_m = 1.0, \quad m=1,2,3.
			\end{aligned}
		\end{equation}
		Furthermore, the gravity constant is set to $g=9.81$ and layer densities to $\rho_1 = 0.9$, $\rho_2 = 0.95$, $\rho_3 = 1.0$ and wall boundary conditions are used. 
		The problem is then computed up to final time $t_{end} = 2.0$ with $\text{CFL}=0.9$ and polynomial degree $N=4$. 
		We set a slightly higher threshold $\tau_{vel}=10^{-6}$ for the velocity desingularization to ensure a robust velocity computation in this complex test case.
		The computation for this test case took around $15.4\,$s of CPU time on our test system.
		
		In Figure~\ref{fig:multilayer_dam_break}, we show the total layer heights and bottom topography at times $t=\{0.1, 1\}$ together with the shock-indicator $\alpha$ and a contour line indicating the wet/dry transition sampled at $\min(h_1,h_2,h_3)\approx 2\cdot10^{-15}$.
		\begin{figure}[htb]
		\centering
		\includegraphics[width=0.9\textwidth, trim={0 2cm 0 5cm}, clip]{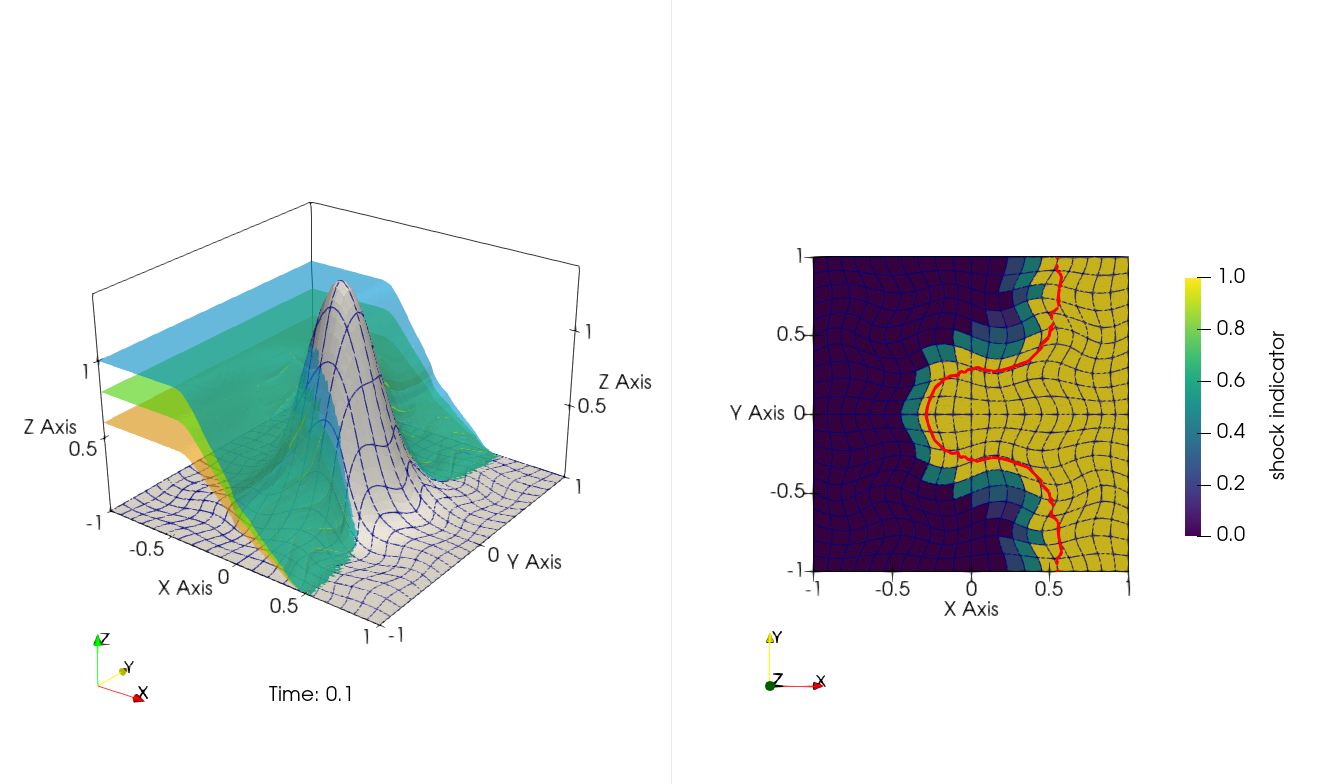}
		\includegraphics[width=0.9\textwidth, trim={0 2cm 0 5cm}, clip]{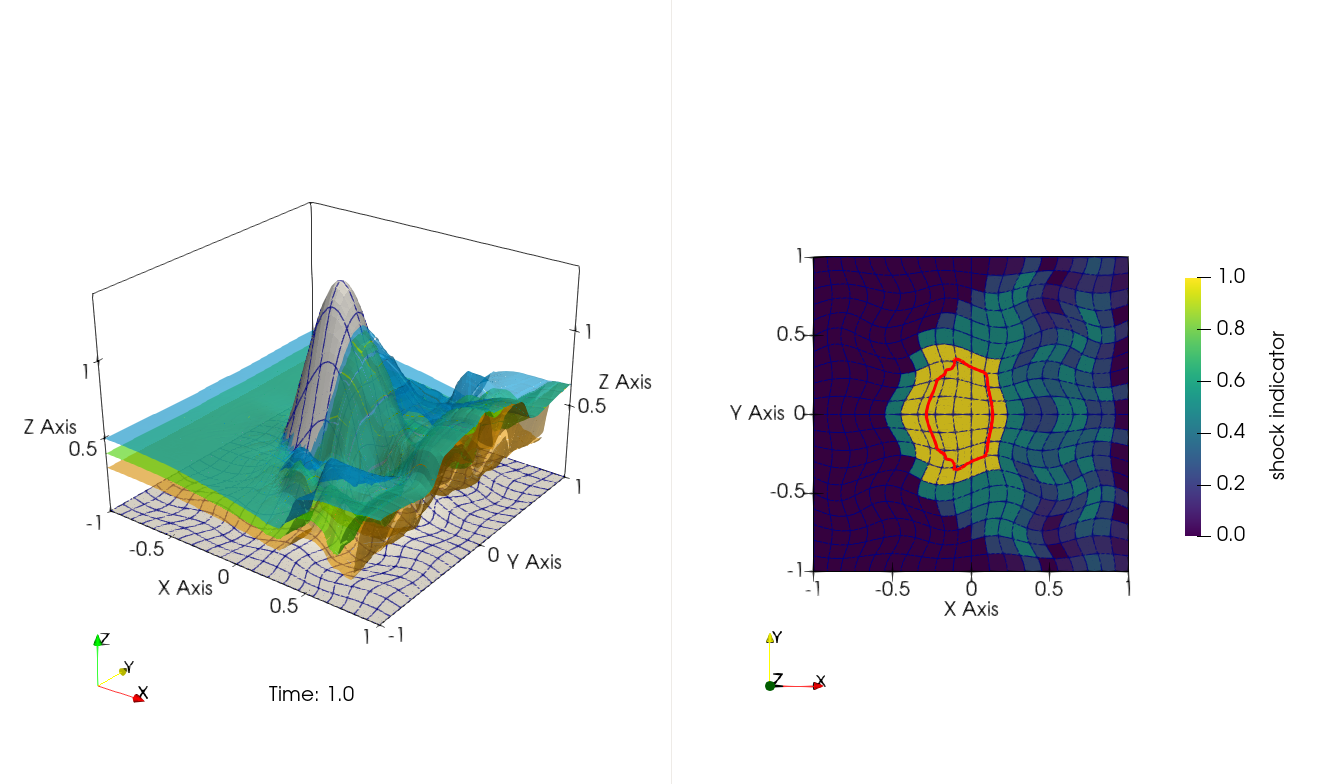}
		\caption{Interface heights (left) and shock indicator with a contour line indicating the wet/dry transition (right) at times $t=\{0.1, 1\}$ obtained with $N=4$ and $\text{CFL}=0.9$.}
		\label{fig:multilayer_dam_break}
		\end{figure}
	
		Initially, the discontinuity of the dam break gives rise to a flood wave over the initially dry domain with a propagating wet/dry transition in all layers, as shown for time $t=0.1$.
		The flood wave is then reflected at the interface located at $x=1$.
		In addition to interactions with the bottom topography this gives rise to intense layer interactions.
		As seen for time $t=1$, this triggers additional shock capturing behind the initial flood wave with indicator values $\alpha=0.001-0.5$, which successfully prevents spurious oscillations.
		In fully wet and smooth regions, the approximation remains fully high-order DG.
		
		In the dry domain and close to the wet/dry interface, we further observe pure FV elements with an indicator value $\alpha=1$ originating from the well-balanced treatment near dry nodes.
		While the FV elements properly capture dry nodes, the element-wise blending on large elements leads to a rather coarse introduction of the low-order method around the transition line.
		
		Even for this complex test case, the method remains robust and preserves positivity of all layer heights throughout the computation, which demonstrates the good properties of the scheme and its compatibility on curvilinear meshes.
		
\section{Conclusion}
	In this work, we introduced a novel hydrostatic reconstruction procedure that can be applied to both shallow water equations (SWE) and multilayer shallow water equations (ML-SWE) to construct well-balanced methods capable of handling wetting and drying.
	
	The new procedure is constructed with a single bottom topography reconstruction and can be applied for SWE and ML-SWE without introducing additional reconstruction terms.
	
	We first introduced a first-order path-conservative finite volume (FV) method and used the new hydrostatic reconstruction with entropy stable fluxes to create a scheme that is well-balanced, entropy stable and positivity-preserving.
	While the reconstruction was found to conserve entropy for the SWE, it introduces a conditional entropy violation for ML-SWE, which we quantified and countered with sufficient numerical dissipation to achieve entropy stability.
	
	We then extended the method to high-order and curvilinear meshes using a collocated nodal discontinuous Galer\-kin spectral element method (DGSEM) on quadrilateral elements.
	Using the summation-by-parts property of the DGSEM, we employ a flux-differencing formulation in the volume, that is applied with entropy conservative two-point fluxes and entropy stable fluxes at interfaces to ensure entropy stability.
	To achieve well-balancedness, we introduced the hydrostatic reconstruction at interfaces and switched to a compatible subcell FV method at dry nodes.
	Positivity was achieved with a positivity-preserving limiter under an additional time-step restriction.
	
	Finally, we verified the theoretical properties of convergence, well-balancedness, entropy stability and positivity-preservation through numerical results. 

\section*{CRediT authorship contribution statement}
	{\bf Patrick Ersing:} Conceptualization; Formal analysis; Investigation; Methodology; Data curation, Visualization, Software; Writing - original draft
	{\bf Andrew R. Winters:} Conceptualization; Funding acquisition; Supervision; Software; Writing - review \& editing
	{\bf Sven Goldberg:} Conceptualization; Methodology; Software; Writing - review \& editing
	
\section*{Data Availability}
	A reproducibility repository with necessary instructions and code to reproduce the presented results is available on Zenodo and GitHub \cite{ersing2024hrRepro}
	
\section*{Declaration of competing interest}
	The authors declare that they have no known competing financial interests or personal relationships that could have appeared to influence the work reported in this paper.
	
\section*{Acknowledgement}	
	Funding: This work was supported by Vetenskapsrådet, Sweden [grant agreement 2020-03642 VR]	

\bibliographystyle{elsarticle-num} 
\bibliography{bibliography}  

\end{document}

%% file: MLSWE_Visu.tex
\begin{tikzpicture}[scale=0.8]

\draw [draw=white, opacity=0.0, name path = x0] (0,-1) -- (6, -1);
\draw [draw=black, name path=bottom] plot [smooth, tension=1] coordinates {(0,0) (1, 0) (2,0.2) (4,0.0) (5, 0.2) (6,0.2)};
\draw [draw=black, name path=layer4] plot [smooth, tension=1] coordinates {(0,1.0) (2, 0.8) (4, 1.0) (6, 0.8)};
\draw [draw=black, name path=layer3] plot [smooth, tension=1] coordinates {(0,2.0) (2, 1.8) (4, 2.0) (6, 1.8)};
\draw [draw=black, name path=layer2] plot [smooth, tension=1] coordinates {(0,3.0) (2, 2.8) (4, 3.0) (6, 2.8)};
\draw [draw=black, name path=layer1] plot [smooth, tension=1] coordinates {(0,4.0) (2, 3.8) (4, 4.0) (6, 3.8)};

\definecolor{earthyellow}{rgb}{0.88, 0.66, 0.37}
\tikzfillbetween[of=layer1 and layer2]{blue, opacity=0.15}
\tikzfillbetween[of=layer2 and layer3]{blue, opacity=0.3}
\tikzfillbetween[of=layer3 and layer4]{blue, opacity=0.45}
\tikzfillbetween[of=bottom and layer4]{blue, opacity=0.6}
\tikzfillbetween[of=x0 and bottom]{earthyellow, opacity=0.9}

\draw [->, -stealth] (0.0,-1) -- (7.0,-1) node[pos=0.5, below]{x};
\draw [->, -stealth] (0.0,-1) -- (0.0, 4.5) node[pos=0.5, left]{z};

\node[xshift=5] at (6.0,0.2) {b};
\node[xshift=7] at (6.0,0.8) {$H_4$};
\node[xshift=7] at (6.0,1.8) {$H_3$};
\node[xshift=7] at (6.0,2.8) {$H_2$};
\node[xshift=7] at (6.0,3.8) {$H_1$};

\draw[<->] (4,0.1) -- (4,0.9) node[pos=0.5, right]{$h_4$};
\draw[<->] (4,1.1) -- (4,1.9) node[pos=0.5, right]{$h_3$};
\draw[<->] (4,2.1) -- (4,2.9) node[pos=0.5, right]{$h_2$};
\draw[<->] (4,3.1) -- (4,3.9) node[pos=0.5, right]{$h_1$};

\draw[->] (1,0.3) -- (1.4,0.3) node[pos=0.5, above]{$v_4$};
\draw[->] (1.4,1.3) -- (1.0,1.3) node[pos=0.5, above]{$v_3$};
\draw[->] (1.4,2.3) -- (1.0,2.3) node[pos=0.5, above]{$v_2$};
\draw[->] (1,3.3) -- (1.4,3.3) node[pos=0.5, above]{$v_1$};

\node at (3,0.5) {$\rho_4$};
\node at (3,1.5) {$\rho_3$};
\node at (3,2.5) {$\rho_2$};
\node at (3,3.5) {$\rho_1$};

\end{tikzpicture}

%% file: ML_norec.tex
\begin{tikzpicture}[scale=0.55]
\path[draw=black, fill=lightgray]
(0,0) -- (10,0) -- 
(10,0) -- (10,5) --
(10,5) -- (5,5) --
(5,5)  -- (5,1) --
(5,1) -- (0,1) --
(0,1) -- (0,0);

\draw (0,4) -- (5,4);
\path[fill=cyan, opacity=0.5] (0,4) -- (5,4) -- (5,1) -- (0,1);
\draw (0,3) -- (5,3);
\path[fill=blue, opacity=0.2] (0,3) -- (5,3) -- (5,2) -- (0,2);
\draw (0,2) -- (5,2);
\path[fill=blue, opacity=0.4] (0,2) -- (5,2) -- (5,1) -- (0,1);

\draw[<->, stealth-stealth, xshift=-0.2cm] (0,1) -- (0,1.95) node [pos=0.5, xshift=-0.3cm]{$h_{1,i}$};
\draw[<->, stealth-stealth, xshift=-0.2cm] (0,2) -- (0,2.95) node [pos=0.5, xshift=-0.3cm]{$h_{2,i}$};
\draw[<->, stealth-stealth, xshift=-0.2cm] (0,3) -- (0,3.95) node [pos=0.5, xshift=-0.3cm]{$h_{3,i}$};


\draw[dashed] (5,0) -- (5,5);
\node[] at (5,-1) [below]{$x_{i+\nicefrac{1}{2}}$};
\node[] at (2.5, 1) [above]{$b_i$};
\node[] at (7.5, 5) [above]{$b_{i+1}$};

\draw [
opacity=0
] (1,0) -- (1.95,0) 
node [pos=0.5,anchor=north,yshift=-0.15cm] {$\epsilon_1$};


\end{tikzpicture}

%% file: ML_reconstructed.tex
\begin{tikzpicture}[scale=0.55]

\path[draw=black, fill=lightgray]
(0,0) -- (10,0) -- 
(10,0) -- (10,5) --
(10,5) -- (5,5) --
(5,5) -- (5,4) --
(5,4) -- (2.5,1) --
(2.5,1) -- (0,1) --
(0,1) -- (0,0);

\draw (0,4) -- (5,4);
\path[fill=cyan, opacity = 0.5] (0,4) -- (5,4) -- (2.5,1) -- (0,1) -- (0,4);
\draw (0,3) -- (4.166,3);
\path[fill=blue, opacity=0.2] (0,3) -- (4.166,3) -- (3.33,2) -- (0,2) -- (0,3);
\draw (0,2) -- (3.33,2);
\path[fill=blue, opacity=0.4] (0,2) -- (3.33,2) -- (2.5,1) -- (0,1) --(0,2);

\draw[<->, stealth-stealth, xshift=-0.2cm] (0,1) -- (0,1.95) node [pos=0.5, xshift=-0.3cm]{$h_{1,i}$};
\draw[<->, stealth-stealth, xshift=-0.2cm] (0,2) -- (0,2.95) node [pos=0.5, xshift=-0.3cm]{$h_{2,i}$};
\draw[<->, stealth-stealth, xshift=-0.2cm] (0,3) -- (0,3.95) node [pos=0.5, xshift=-0.3cm]{$h_{3,i}$};

\node[] at (1.5,1) [above left]{$b_i$};
\node (urec-) at (5,4) {\textbullet};
\node[] at (5,4) [above left]{$b_{i+\nicefrac{1}{2},\epsilon}^-$};
\node (urec+) at (5,5) {\textbullet};
\node[] at (5,5) [above right]{$b_{i+\nicefrac{1}{2},\epsilon}^+$};
\node[] at (8.25,5) [above right]{$b_{i+1}$};

\draw[dotted, thick] (5,0) -- (5,4);
\node[] at (5,-1) [below] {$x_{i+\nicefrac{1}{2}}$};
\draw[dotted, thick] (2.5,0) -- (2.5,4);
\node[] at (2.5,-1) [below] {$x_{i+\nicefrac{1}{2}}-\epsilon$};
\draw[dotted, thick] (7.5,0) -- (7.5,5);
\node[] at (7.5,-1) [below] {$x_{i+\nicefrac{1}{2}}+\epsilon$};

\draw [
thick,
decoration={
	brace,
	mirror,
	raise=0.1cm
},
decorate
] (2.55,0) -- (3.28,0) 
node [pos=0.5,anchor=north,yshift=-0.15cm] {$\epsilon_1$};
\draw [
thick,
decoration={
	brace,
	mirror,
	raise=0.1cm
},
decorate
] (3.38,0) -- (4.2,0) 
node [pos=0.5,anchor=north,yshift=-0.15cm] {$\epsilon_2$};
\draw [
thick,
decoration={
	brace,
	mirror,
	raise=0.1cm
},
decorate
] (4.25,0) -- (4.95,0) 
node [pos=0.5,anchor=north,yshift=-0.15cm] {$\epsilon_3$};

\draw[dotted] (3.33,0) -- (3.33,4);
\draw[dotted] (4.166,0) -- (4.166,4);

\end{tikzpicture}

%% file: Wet_case.tex
\begin{tikzpicture}[scale=0.4]
	
\path[fill=cyan, opacity = 0.5] (0,6.5) -- (8,6.5) -- (8,5) -- (4,5) -- (1,1) -- (0,1) -- (0,6.5);

\path[draw=black,dashed,color=black, pattern = crosshatch dots, pattern color=red] (1,1) -- (4,5) -- (4,1) -- (1,1);
	
\path[draw=black,pattern=north west lines]
(0,0) -- (8,0) -- 
(8,0) -- (8,5) --
(8,5) -- (4,5) --
(4,5) -- (4,1) --
(4,1) -- (1,1) --
(1,1) -- (0,1) --
(0,1) -- (0,0);

\draw (0,6.5) -- (8,6.5);

\draw[dashed] (1,1) -- (4,5);

\node (uminus) at (1,1) {\textbullet};
\node[] at (1,1) [above]{$b^-$};
\node (urec-) at (4,5) {\textbullet};
\node[] at (4,5) [above]{$b_{\epsilon}^- = b_{\epsilon}^+$};

\node (uplus) at (7,5) {\textbullet};
\node[] at (7,5) [above]{$b^+$};

\end{tikzpicture}

%% file: Dry_Case.tex
\begin{tikzpicture}[scale=0.4]

\path[fill=cyan, opacity = 0.5] (0,4) -- (4,4) -- (1,1) -- (0,1) -- (0,4);

\path[draw=black,dashed,color=black, pattern = crosshatch dots, pattern color=red] (1,1) -- (4,4) -- (4,1) -- (1,1);
	
\path[draw=black,pattern=north west lines]
(0,0) -- (8,0) -- 
(8,0) -- (8,5) --
(8,5) -- (4,5) --
(4,5) -- (4,1) --
(4,1) -- (1,1) --
(1,1) -- (0,1) --
(0,1) -- (0,0);

\draw (0,4) -- (4,4);

\draw[dashed] (1,1) -- (4,4);

\node (uminus) at (1,1) {\textbullet};
\node[] at (1,1) [above left]{$b^-$};
\node (urec-) at (4,4) {\textbullet};
\node[] at (4,4) [above left]{$b_{\epsilon}^-$};
\node (uplus) at (7,5) {\textbullet};
\node (urec+) at (4,5) {\textbullet};
\node[] at (4,5) [above right]{$b_{\epsilon}^+$};
\node[] at (7,5) [above]{$b^+$};

\end{tikzpicture}

%% file: Partial_Dry_Case.tex
\begin{tikzpicture}[scale=0.4]

\path[fill=cyan, opacity = 0.5] (0,4) -- (4,4) -- (1,1) -- (0,1) -- (0,4);

\path[draw=black,dashed,color=black, pattern = crosshatch dots, pattern color=red] (1,1) -- (4,4) -- (4,1) -- (1,1);

\path[fill=cyan, opacity = 0.5] (4,5) -- (4,6.5) -- (8,6.5) -- (8,5) -- (4,5);

\path[draw=black,pattern=north west lines]
(0,0) -- (8,0) -- 
(8,0) -- (8,5) --
(8,5) -- (4,5) --
(4,5) -- (4,1) --
(4,1) -- (1,1) --
(1,1) -- (0,1) --
(0,1) -- (0,0);

\draw (0,4) -- (4,4);
\draw (4,4) -- (4,6.5);
\draw (4,6.5) -- (8,6.5);

\draw[dashed] (1,1) -- (4,4);

\node (uminus) at (1,1) {\textbullet};
\node[] at (1,1) [above]{$b^-$};
\node (urec-) at (4,4) {\textbullet};
\node[] at (4,4) [above left]{$b_{\epsilon}^-$};
\node (uplus) at (7,5) {\textbullet};
\node (urec+) at (4,5) {\textbullet};
\node[] at (4,5) [above right]{$b_{\epsilon}^+$};
\node[] at (7,5) [above]{$b^+$};

\end{tikzpicture}